\documentclass[sn-mathphys-num]{sn-jnl}

\usepackage{arydshln}
\usepackage{graphicx}%
\usepackage{multirow}%
\usepackage{amsmath,amssymb,amsfonts}%
\usepackage{amsthm}%
\usepackage{mathrsfs}%
\usepackage[title]{appendix}%
\usepackage{xcolor}%
\usepackage{textcomp}%
\usepackage{manyfoot}%
\usepackage{booktabs}%
\usepackage{algorithm}%
\usepackage{algorithmicx}%
\usepackage{algpseudocode}%
\usepackage{listings}%
\allowdisplaybreaks[4]
 \usepackage{geometry}
\def\theequation{\arabic{section}.\arabic{equation}}
\newtheorem{Theorem}{Theorem}[section]
\newtheorem{Assumption}[Theorem]{Assumption}
\newtheorem{Lemma}[Theorem]{Lemma}
\newtheorem{Corollary}[Theorem]{Corollary}

\newtheorem{Remark}[Theorem]{Remark}

\theoremstyle{thmstylethree}%

\begin{document}

\title[ ]{Strong convergence of a fully discrete scheme for stochastic Burgers equation with fractional-type noise}


\author[1]{\fnm{Yibo} \sur{Wang}}\email{wangyb@seu.edu.cn}

\author*[1]{\fnm{Wanrong} \sur{Cao}}\email{wrcao@seu.edu.cn}

\affil[1]{\orgdiv{School of Mathematics}, \orgname{Southeast University}, \orgaddress{ \city{Nanjing}, \postcode{210096}, \country{People's Republic of China}}}


\abstract{We investigate numerical approximations for the stochastic Burgers equation driven by an additive cylindrical fractional Brownian motion with Hurst parameter $H \in (\frac{1}{2}, 1)$. To discretize the continuous problem in space, a spectral Galerkin method is employed, followed by the presentation of a nonlinear-tamed accelerated exponential Euler method to yield a fully discrete scheme. By showing the exponential integrability of the stochastic convolution of the fractional Brownian motion, we present the boundedness of moments of semi-discrete and full-discrete approximations. Building upon these results and the convergence of the fully discrete scheme in probability proved by a stopping time technique, we derive the strong convergence of the proposed scheme. }




\maketitle

\section{Introduction}	
The fractional Brownian motion (fBm)  is a centered Gaussian process characterized by stationary increments and self-similarity with Hurst parameter $H \in (0,1)$ \cite{Mandelbrot1968,Biagini2008,Duncan2002}. Due to its capability to depict correlated fluctuations,  the fBm finds appreciable applications in various fields like finance, biology, and communication engineering. In particular, it was introduced by Kolmogorov \cite{Kolmogorov1940} to model turbulence in liquids, and 
it was later incorporated in a general formalism to describe turbulent wave-front phase \cite{Perez2004}.  When $\frac{1}{2} < H < \frac{5}{6}$, the formalism extends to non-Kolmogorov turbulence. Meanwhile, the stochastic Burgers equation (SBE) was discovered to have statistical characteristics similar to experimental investigations of turbulence \cite{Chekhlov1995,Jeng1969}. 
The inclusion of the adjustable Hurst parameter in the SBE endows it with a powerful ability to model complex turbulence phenomena, and thus, the SBE driven by fBm  has aroused widespread research interest.

Denote $\mathcal{D}:= (0,1)$ and $U := L^2(\mathcal{D})$. In this paper, we focus on numerical methods for the following SBE driven by a cylindrical fBm
\begin{equation}\label{Burgers}
	\left\{\begin{aligned}
		&d u(t) = \left( -Au(t) + f(u(t)) \right) dt + d B^H(t), \quad t \in (0, T], \\
		&u(0) = u_0 \in U, 
	\end{aligned}\right.
\end{equation}
where $-A: \text{dom}(A) \subset U \rightarrow U$ is the Laplacian with homogeneous Dirichlet boundary conditions, $f: L^4(\mathcal{D}) \rightarrow H^{-1}(\mathcal{D})$ is a deterministic mapping given by $f(u)(x) := u(x) \frac{\partial}{\partial x} u(x)$, and $B^H(t)$ is a standard cylindrical fBm on the stochastic basis $(\Omega, \mathcal{F}, \mathcal{F}_t, \mathbb{P})$, defined by a formal series
\begin{equation}\label{fBm}
	B^H(t) := \sum_{k=1}^{\infty} w_k^H(t) \phi_k (x) , \quad t \in [0,T].
\end{equation}
In \eqref{fBm}, $\{ w_k^H(t) \}_{k\in\mathbb{N}}$ is a sequence of independent real-valued standard fBms with the same Hurst parameter $H \in (\frac{1}{2}, 1)$,  and $\{\phi_{k}\}_{k \in \mathbb{N}}$ is an orthonormal basis of the Hilbert space $L^2(\mathcal{D})$. Note that the scalar fBm $w_k^H(t)$ with $H > \frac{1}{2}$ is neither a Markov process	nor a semi-martingale, but a centered Gaussian process with continuous sample paths and covariance function
$R(t,s) = \frac{1}{2} \left(s^{2H} + t^{2H} - |t-s|^{2H} \right)$. 
If taking $H=\frac{1}{2}$, the cylindrical fBm $B^H(t)$ becomes the cylindrical Brownian motion that has been widely used to describe uncorrelated random perturbations.

The well-posedness has been established for the SBE with cylindrical Brownian motion \cite{Prato1995,Prato1994} and with cylindrical fBm \cite{Wang2010}. However, the solution of the SBE is a stochastic field that can hardly be given explicitly,  so it is crucial to develop an effective numerical method to solve \eqref{Burgers}.  Various numerical methods such as the finite difference method \cite{Hairer2011,Alabert2006} and the Galerkin approximation \cite{Blomker2013-SIAM,Blomker2013-IMA,Khan2021} have been considered in recent years for SBEs driven by cylindrical Brownian motion. Nevertheless, the convergence in the aforementioned works was analyzed in the pathwise sense. In general, convergence in the strong sense is a more powerful and natural measure to express the convergence properties of numerical solutions by averaging the pathwise error over all sample paths. To the best of our knowledge, there has been little work on the strong convergence of numerical methods for SBEs.  In previous studies, the strong convergence of a truncated exponential Euler-type scheme and a Wong--Zakai approximation for the SBE driven by cylindrical Brownian motion were proposed in \cite{Jentzen2017} and \cite{Gugg2002}, respectively. Additionally, for the SBEs with additive and multiplicative trace-class noises, the exponential integrability and strong convergence rates of some fully discrete schemes were obtained \cite{Hutzenthaler2019,Hutzenthaler2022}.

The global monotonicity condition is vital to the strong convergence analysis for numerical schemes. Several fruitful results have been achieved so far on numerical methods for stochastic partial differential equations (SPDEs) with monotone drift. 
On this topic, we refer to \cite{Becker2023,Becker2019,Brehier2019,Cai2021,Wang2020,Liu2021} and references therein. Unfortunately, the nonlinear term $\displaystyle uu_x$ in the SBE, and also in other SPDEs such as the stochastic Kuramoto--Sivashinsky equation \cite{Hutzenthaler2018,Wu2018} and the stochastic Navier--Stokes equation \cite{Carelli2012,Kotelenez1995}, does not fall into this category. This factor poses significant challenges to the strong convergence analysis of the numerical methods.  

Given significant applications in turbulence modeling, this paper aims to construct an explicit fully discrete scheme for the SBE driven by a cylindrical fBm \eqref{Burgers} with a particular focus on the strong convergence of the numerical method. Specifically, the spectral Galerkin method is utilized for spatial discretization and a tamed accelerated exponential Euler method is used for temporal discretization. To prove the strong convergence of the proposed scheme, the convergence in probability is first examined using the stopping time technique. Then, proving the strong convergence is converted to show the moment boundedness of numerical solutions. The regularity properties of the stochastic convolution of cylindrical fBm are presented in Section \ref{Sec:semi-discrete}, drawing inspiration from the Kolmogorov test in \cite{Prato2014}. This, together with Fernique's theorem \cite{Stroock2011}, is used to develop the exponential integrability of the stochastic convolution and the moment boundedness of numerical solutions.


The paper is structured as follows. In Section \ref{Sec:preliminary}, we introduce some important notions and estimates that will be used later. Additionally, we provide some pathwise estimates for the mild solution of the SBE. Section \ref{Sec:semi-discrete} presents the spatial discretization of the continuous problem using the spectral Galerkin method. Furthermore, we discuss the regularity properties of the stochastic convolution, in particular, the exponential integrability. The strong convergence of the semi-discretization is also proved in this section, followed by the strong convergence of the full discretization in Section \ref{Sec:full-discrete}. We present a numerical experiment in Section \ref{Sec:numerical example} to verify the theoretical results. Lastly, we will summarize and discuss the findings in Section \ref{Sec:conclusion}.

\section{Preliminary}\label{Sec:preliminary}
Recall that $\mathcal{D}:= (0,1)$. 
We denote by $\left( L^p(\mathcal{D}; \mathbb{R}), \|\cdot\|_{L^p(\mathcal{D})} \right)$, $1 \leq p \leq \infty$, or shortly, $\left( L^p(\mathcal{D}), \|\cdot\|_{L^p} \right)$, the Banach space consisting of all $p$-times integrable functions. 
When $p=\infty$, $L^\infty(\mathcal{D})$ contains all functions which are essentially bounded.  
When $p=2$, we denote by $U$ the Hilbert space $L^2(\mathcal{D})$ with the scalar product $\langle \cdot, \cdot \rangle$ and the norm $\|\cdot\|$. 
For a separable Hilbert space $\left( \mathcal{V}, \langle \cdot , \cdot \rangle_{\mathcal{V}}, \|\cdot\|_{\mathcal{V}} \right)$, we denote by $\mathcal{L}(\mathcal{V})$ the space consisting of all bounded linear operators on $\mathcal{V}$ endowed with the induced operator norm $\|\cdot\|_{\mathcal{L}(\mathcal{V})}$. 
A bounded linear operator $T \in \mathcal{L}(\mathcal{V})$ is called Hilbert--Schmidt if $\sum_{k \in \mathbb{N}} \| T \psi_k \|^2_{\mathcal{V}} < \infty$, where $\{\psi_k\}_{k \in \mathbb{N}}$ is an arbitrary orthonormal basis of $\mathcal{V}$. 
The space of all Hilbert--Schmidt operators on $\mathcal{V}$, denoted by $\mathcal{L}_2(\mathcal{V})$, equipped with the norm $\|T\|_{\mathcal{L}_2(\mathcal{V})} := \left( \sum_{k \in \mathbb{N}} \| T \psi_{k} \|^2_{\mathcal{V}} \right)^{\frac{1}{2}}$, is a separable Hilbert space with the scalar product $\langle T_1, T_2 \rangle_{\mathcal{L}_2(\mathcal{V})} := \sum_{k \in \mathbb{N}} \langle T_1 \psi_{k} , T_2 \psi_{k} \rangle_{\mathcal{V}}$. 

In what follows by $C$ we denote the generic positive constant which may be different from line to line but is always independent of the discretization parameters. Sometimes we use $C(\cdot)$ to emphasize its dependence on the parameters in brackets. 

There exist an increasing sequence of positive numbers $\lambda_{k}$, such that $A \phi_k = \lambda_k \phi_k$, where $\{\phi_{k}\}_{k \in \mathbb{N}}$ form an orthonormal basis of $U$. 
By using the fractional power of $A$, we introduce the Hilbert space $\dot{H}^\gamma$, $\gamma \in \mathbb{R}$, endowed with the inner product $\langle \cdot , \cdot \rangle_{\gamma} := \langle A^{\frac{\gamma}{2}} \cdot , A^{\frac{\gamma}{2}} \cdot \rangle$ and the norm $\| \cdot \|_{\dot{H}^\gamma} := \langle \cdot , \cdot \rangle_{\gamma}^{\frac{1}{2}}$. 
The operator $-A$ generates an analytic semigroup $S(t) = \exp(-tA)$, which satisfies 
\begin{equation}\label{Est:semigroup 1}
	\begin{aligned}	
		\left\| A^{\gamma} S(t) \right\|_{\mathcal{L}(U)} &\leq C t^{-\gamma} , \quad t > 0, \ \gamma \geq 0,  \\
		\left\| A^{-\rho} (I - S(t)) \right\|_{\mathcal{L}(U)} &\leq C t^{\rho} , \quad t > 0, \ \rho \in [0,1]. 
	\end{aligned}
\end{equation}
In addition, for any $s_1 \leq s_2$ in $\mathbb{R}$ and $r \geq 1$, it holds that 
\begin{equation}\label{Est:semigroup 2}
	\| S(t)u \|_{W^{s_2 , r}} \leq C t^{\frac{s_1-s_2}{2}} \| u \|_{W^{s_1 , r}},  \quad \text{for all} \  u \in W^{s_1 , r}. 
\end{equation}
The constant $C$ in \eqref{Est:semigroup 2} depends only on $s_1$, $s_2$, and $r$ (see, e.g., \cite{Rothe1984,Prato1994}). For $\zeta \in (0, \frac{1}{2})$, by the semigroup property of $S(t)$ and the Sobolev embedding inequality $W^{\frac{1}{2}, 1}(\mathcal{D}) \hookrightarrow L^2(\mathcal{D})$, taking $\gamma = \frac{\zeta}{2}$ in \eqref{Est:semigroup 1} and $s_1 = -1$, $s_2 = \frac{1}{2}$, $r = 1$ in \eqref{Est:semigroup 2}, we have
\begin{equation}\label{Est:S B 1}
	\|S(t)u\|_{\dot{H}^{\zeta}} 
	= \left\|A^{\frac{\zeta}{2}} S\left(\tfrac{t}{2}\right) S\left(\tfrac{t}{2}\right) u \right\| 
	\leq C(\zeta) t^{-\frac{\zeta}{2}} \left\| S\left(\tfrac{t}{2}\right) u \right\|_{W^{\frac{1}{2}, 1}} \leq C(\zeta) t^{-\frac{\zeta}{2}-\frac{3}{4}} \left\| u \right\|_{W^{-1, 1}}. 
\end{equation} 
Similarly, taking $\gamma = \frac{\zeta}{2}$ in \eqref{Est:semigroup 1} and $s_1 = -1$, $s_2 = 0$, $r = 2$ in \eqref{Est:semigroup 2}, we get
\begin{equation}\label{Est:S B 2}
	\|S(t)u\|_{\dot{H}^{\zeta}} 
	= \left\|A^{\frac{\zeta}{2}} S\left(\tfrac{t}{2}\right) S\left(\tfrac{t}{2}\right) u \right\| 
	\leq C(\zeta) t^{-\frac{\zeta}{2}} \left\| S\left(\tfrac{t}{2}\right) u \right\|
	\leq C(\zeta) t^{-\frac{\zeta}{2} -\frac{1}{2}} \left\| u \right\|_{W^{-1, 2}}.  
\end{equation} 
By denoting $\phi (y) := H(2H-1) |y|^{2H-2}$, $ y \in \mathbb{R}$, we quote from \cite[Lemma 3.6]{Wang2017} the following lemma related to the regularity property of the semigroup $S(t)$. 
\begin{Lemma}\label{Le:S(t)}
	There exists a constant $C$ only depending on the parameter $H$, such that  
	\begin{equation*}
		\int_{s}^{t} \int_{s}^{t} \left\langle A^{\vartheta} S(t-\mu) \psi , A^{\vartheta} S(t-\nu) \psi \right\rangle \phi(\mu-\nu) d\mu d\nu 
		\leq C (t-s)^{2(H-\vartheta)} \|\psi\|^2 , \quad \vartheta \in [0,H]  
	\end{equation*}
	holds for $0 < s \leq t$ and $\psi \in U$. 
\end{Lemma}

Next, we introduce the definition of stochastic integral with respect to the cylindrical fBm \eqref{fBm} that was used in \cite{Wang2017,Duncan2002}.
Assume that $g(\cdot) \psi \in L^p(0,T;U)$, $\forall \, \psi \in U$, $p \geq 2$, and 
\begin{equation*}
	\int_{0}^{T} \int_{0}^{T} \|g(u)\|_{\mathcal{L}_2(U)}  \|g(v)\|_{\mathcal{L}_2(U)}   \phi(u-v) du dv < \infty. 
\end{equation*}
The stochastic integral is defined by 
\begin{equation}\label{def:int fBm}
	\int_{0}^{T} g(s) d B^H(s) := \sum_{k=1}^{\infty} \int_{0}^{T} g(s) \phi_k d w_k^H(s), 
\end{equation}
where the summation is  taken in the mean-square sense. 
The series in \eqref{def:int fBm} is a zero mean, $U$-valued Gaussian random variable and satisfies the following It\^o's isometry: 
\begin{equation*}
	\mathbb{E} \left[\left\| \int_{0}^{T} g(s) dB^H(s) \right\|^2\right] = \int_{0}^{T} \int_{0}^{T} \left\langle g(u) , g(v) \right\rangle_{\mathcal{L}_2(U)} \phi(u-v) du dv. 
\end{equation*}
We denote by $\mathcal{O}_t$, $t \in [0,T]$ the stochastic convolution of the cylindrical fBm, given by 
\begin{equation*}
	\mathcal{O}_t := \int_{0}^{t} S(t-s) d B^H(s).  
\end{equation*}
\begin{Assumption}\label{Asp:initial value}
	Let the initial value $u_0$ be an $\mathcal{F}_0/\mathcal{B(U)}$-measurable random variable. Assume that for $\frac{1}{4} \leq \varrho \leq \frac{3}{8}$, $\varrho - 2\theta \geq \frac{1}{4}$, $\theta < \frac{1}{8}$, and sufficiently large positive integer $p$, 
	\begin{equation*}
		\mathbb{E} \left[\|u_0\|^p\right] + \mathbb{E} \left[\|u_0\|^2_{\dot{H}^{\varrho}}\right] + \mathbb{E} \left[\|u_0\|^8_{\dot{H}^{1/4}}\right] + \mathbb{E} \left[\|u_0\|^4_{\dot{H}^{1/4+4\theta}}\right] < \infty. 
	\end{equation*}
\end{Assumption}
The following theorem, shown in \cite{Wang2010}, presents the well-posedness of the SBE \eqref{Burgers}. 
\begin{Theorem}
	Under Assumption \ref{Asp:initial value}, the SBE \eqref{Burgers} possesses a unique mild solution $u \in C([0,T]; L^p(\mathcal{D}))$, $p \geq 2$, determined by 
	\begin{equation}\label{Eq:mild solution u1}
		u(t) = S(t) u_0 + \int_{0}^{t} S(t-s) f(u(s)) ds + \mathcal{O}_t, \quad \mathbb{P}\text{-a.s.}. 
	\end{equation}
\end{Theorem}

In the remainder of this section, we provide pathwise estimates for the mild solution of the SBE. 
For $\eta \geq 0$, we denote  
\begin{equation}\label{Eq:O eta}
	\mathbb{O}^{\eta}_t := \int_{0}^{t} S^{\eta}(t-s) d B^H(s) = \int_{0}^{t} e^{-(t-s) (A+\eta)} d B^H(s), 
\end{equation}
where $S^{\eta}(t)$ is the semigroup generated by $-A-\eta$. 
For a given $\eta \geq 0$, the SBE \eqref{Burgers} is equivalent to 
\begin{equation*}
	d u(t) = \left( (-A-\eta) u(t) + f(u(t)) + \eta u(t) \right) dt + d B^H(t), 
\end{equation*}
which possesses a mild solution determined by  
\begin{equation*}
	u(t) = S^{\eta}(t) u_0 + \int_{0}^{t} S^{\eta}(t-s) f(u(s)) ds + \eta \int_{0}^{t} S^{\eta}(t-s) u(s) ds + \mathbb{O}^{\eta}_t. 
\end{equation*}
By introducing $\tilde{u}_{\eta}(t) := u(t) - \mathbb{O}^{\eta}_t$, we obtain 
\begin{equation*}
	\tilde{u}_\eta (t) = S^{\eta}(t) u_0 + \int_{0}^{t} S^{\eta}(t-s) f(\tilde{u}_\eta(s) + \mathbb{O}^{\eta}_s) ds + \eta \int_{0}^{t} S^{\eta}(t-s) \left( \tilde{u}_\eta(s) + \mathbb{O}^{\eta}_s \right) ds, 
\end{equation*}
which means that $\tilde{u}_\eta(t): [0,T] \times \Omega \rightarrow U$ pathwisely solves
\begin{equation}\label{Eq:u tilde eta}
	\begin{aligned}
		\frac{d}{dt} \tilde{u}_\eta(t) &= (-A-\eta) \tilde{u}_\eta(t) + f(\tilde{u}_\eta(t) + \mathbb{O}^{\eta}_t) + \eta \left( \tilde{u}_\eta(t) + \mathbb{O}^{\eta}_t \right) \\
		&= -A \tilde{u}_\eta(t) + f(\tilde{u}_\eta(t) + \mathbb{O}^{\eta}_t) + \eta \mathbb{O}^{\eta}_t. 
	\end{aligned}
\end{equation}
\begin{Remark}
	We note that the equation \eqref{Eq:u tilde eta} is only a formal way of writing, since $-A \tilde{u}_\eta$ has no meaning. In fact, replacing by a smooth approximation, i.e. by a smooth process of $\mathbb{O}^{\eta}_t$ in \eqref{Eq:u tilde eta}, then this equation, and the subsequent derivations are meaningful for any smooth approximations and therefore for $\tilde{u}_\eta(t)$ itself by taking limits. This density argument can be found in, e.g. \cite{Prato1994} and references therein. Instead, here we omit this lengthy step for convenience. 
	
	That we additionally introduce a stochastic convolution $\mathbb{O}^{\eta}_t$, i.e. \eqref{Eq:O eta}, is because the exponential moment of $\mathbb{O}^{\eta}_t$ can be estimated for a proper $\eta \geq 0$ (see Lemma \ref{Le:exponential int O}) while the estimate of exponential moment of the stochastic convolution $\mathcal{O}_t$ is not an easy work. 
\end{Remark}
By using the identity $\frac{1}{2} \frac{d}{dt} \| \tilde{u}_\eta(t) \|^2 = \langle \tilde{u}_\eta(t) , \frac{d}{dt} \tilde{u}_\eta(t) \rangle$, the integral by parts,  $\int_{\mathcal{D}} \tilde{u}_\eta^2 \frac{\partial}{\partial x} \tilde{u}_\eta dx = 0$, $ab \leq a^2 + \frac{1}{4}b^2$ and $ab \leq \frac{1}{2} a^2 + \frac{1}{2} b^2$, it is easy to verify 
\begin{equation*}
	\frac{d}{dt} \| \tilde{u}_\eta(t) \|^2 + \dfrac{1}{2} \| \tilde{u}_\eta(t) \|_{\dot{H}^1}^2
	\leq \left(\| \mathbb{O}^{\eta}_t \|_{L^\infty}^2 + 1 \right)  \| \tilde{u}_\eta(t)\|^2  
	+ \frac{1}{2} \| \mathbb{O}^{\eta}_t \|_{L^4}^4 
	+ \eta^2 \| \mathbb{O}^{\eta}_t \|^2. 
\end{equation*}
By using Gronwall's inequality and noting that $\tilde{u}_\eta(0) = u_0$, we obtain 
\begin{equation}\label{Est:u eta}
	\| \tilde{u}_\eta(t) \|^2 
	\leq \left( \|u_0\|^2 + \frac{1}{2} \int_{0}^{t} \| \mathbb{O}^{\eta}_s \|_{L^4}^4 ds
	+ \eta^2 \int_{0}^{t} \| \mathbb{O}^{\eta}_s \|^2 ds \right)
	e^{\int_{0}^{t} \left(\| \mathbb{O}^{\eta}_s \|_{L^\infty}^2 + 1 \right) ds} . 
\end{equation}
As a by-product, we can derive 
\begin{equation}\label{Est:int u eta}
	\frac{1}{2} \int_{0}^{t} \| \tilde{u}_\eta(s) \|_{\dot{H}^1}^2 ds
	\leq \|u_0\|^2 + \int_{0}^{t} \left(\| \mathbb{O}^{\eta}_s \|_{L^\infty}^2 + 1 \right)  \| \tilde{u}_\eta(s)\|^2  ds 
	+ \frac{1}{2} \int_{0}^{t} \| \mathbb{O}^{\eta}_s \|_{L^4}^4 ds
	+ \eta^2 \int_{0}^{t} \|\mathbb{O}^{\eta}_s\|^2 ds. 
\end{equation}

\section{Spatial discretization}\label{Sec:semi-discrete}
In this section we conduct the spatial discretization by using the spectral Galerkin method. 
Introduce the spectral Galerkin projection $P_N: U \rightarrow U^N := \text{span} \{ \phi_1 , \phi_2 , ... , \phi_N \}$ and the operator $A_N := P_N A$, $N \in \mathbb{N}$. 
It is easy to check that 
\begin{equation}\label{projection}
	\|(P_N-I)\psi\| \leq \lambda_{N}^{-\frac{\alpha}{2}} \|\psi\|_{\dot{H}^{\alpha}} , \quad \forall \, \psi \in \dot{H}^{\alpha} , \ \alpha \geq 0. 
\end{equation}
The aim of spatial discretization is to construct $u^N : [0,T] \times \Omega \rightarrow U^N$ solving the following finite dimensional equation $\mathbb{P}$-a.s.. 
\begin{equation}\label{Eq:semi discrete}
	d u^N(t) = \left( -A_N u^N(t) + P_N f(u^N(t)) \right) dt + P_N dB^H(t). 
\end{equation}
The equation \eqref{Eq:semi discrete} possesses a unique mild solution, given by 
\begin{equation}\label{Eq:mild solution uN}
	u^N(t) = S_N(t) P_N u_0 + \int_{0}^{t} S_N(t-s) P_N f(u^N(s)) ds + \int_{0}^{t} S_N(t-s) P_N d B^H(s), 
\end{equation}
where $S_N(t) := \exp(-t A_N)$ is the semigroup generated by $-A_N$. 
Similarly to the continuous problem, for $\eta \geq 0$, we can rewrite \eqref{Eq:semi discrete} as 
\begin{align*}
	d u^N(t) = \left( (-A_N-\eta) u^N(t) + P_N f(u^N(t)) + \eta u^N(t) \right) dt + P_N d B^H(t), 
\end{align*}
and pass to the mild solution   
\begin{align*}
	u^N(t) = S_N^\eta(t) P_N u_0 + \int_{0}^{t} S_N^\eta(t-s) \left( P_N f(u^N(s)) + \eta u^N(s) \right) ds + \int_{0}^{t} S_N^\eta(t-s) P_N dB^H(s),  
\end{align*}
where $S_N^\eta(t) :=  \exp(-t(A_N+\eta))$ is the semigroup generated by $-A_N - \eta$. 
We denote 
$$\mathcal{O}^N_t := P_N \mathcal{O}_t := \int_{0}^{t} S_N(t-s) P_N d B^H(s), $$
and for $\eta \geq 0$, 
$$\mathbb{O}^{N, \eta}_t := \int_{0}^{t} S_N^\eta(t-s) P_N d B^H(s) = \int_{0}^{t} e^{-(t-s) (A_N+\eta)} P_N d B^H(s). $$
Introduce $\tilde{u}^N_{\eta}(t) := u^N(t) - \mathbb{O}^{N, \eta}_t$, which solves 
\begin{equation*}
	\frac{d}{dt} \tilde{u}^N_\eta(t) 
	= -A_N \tilde{u}^N_\eta(t) + f(\tilde{u}^N_\eta(t) + \mathbb{O}^{N, \eta}_t) + \eta \mathbb{O}^{N, \eta}_t. 
\end{equation*}
An analogous estimate of $\tilde{u}_\eta$, i.e. \eqref{Est:u eta}, leads to the following pathwise estimate of $\tilde{u}^N_\eta(t)$: 
\begin{equation}\label{Est:u eta N}
	\| \tilde{u}^N_\eta(t) \|^2 
	\leq \left( \|P_N u_0\|^2 + \frac{1}{2} \int_{0}^{t} \| \mathbb{O}^{N, \eta}_s \|_{L^4}^4 ds
	+ \eta^2 \int_{0}^{t} \| \mathbb{O}^{N, \eta}_s \|^2 ds \right)
	e^{\int_{0}^{t} \left(\| \mathbb{O}^{N, \eta}_s \|_{L^\infty}^2 + 1 \right) ds} . 
\end{equation}

\subsection{Estimate of the stochastic convolution}
In the following we give some properties of stochastic convolutions. 
We consider the general case, i.e., to estimate $\mathbb{O}^{\eta}_t$ for $\eta \geq 0$. 
Other special cases are included, e.g., the estimates of $\mathcal{O}_t$ and $\mathcal{O}^N_t$ follow from the case of $\eta = 0$ and the finite $N$-terms summation.  

\begin{Lemma}\label{Le:O eta}
	Let $\beta \in (0,\frac{1}{4})$, $p \in (\frac{1}{\beta} , \infty)$ and $H \in (\frac{1}{2}, 1)$. For $\eta \geq 0$, it holds that  
	\begin{equation*}
		\left(\mathbb{E} \left[\| \mathbb{O}^{\eta}_t \|_{L^\infty}^p\right] \right)^{\frac{1}{p}}
		\leq C(p, \beta, H) \left( \sum_{k=1}^{\infty} \frac{ k^{4\beta} }{(\lambda_k+\eta)^{2H}} \right)^{\frac{1}{2}}. 
	\end{equation*}
\end{Lemma}
\begin{proof}
	By using the Sobolev embedding inequality $W^{\beta , p}(\mathcal{D}) \hookrightarrow L^\infty(\mathcal{D})$, where $\beta p > 1$ and $p \in (\frac{1}{\beta} , \infty)$, we have 
	\begin{equation}\label{Est:Sobolev of O}
		\mathbb{E} \left[\| \mathbb{O}^{\eta}_t \|_{L^\infty}^p\right] \leq C \mathbb{E} \left[\| \mathbb{O}^{\eta}_t \|_{W^{\beta , p}}^p\right].  
	\end{equation}
	Through the intrinsic norm of $W^{\beta , p}(\mathcal{D})$, and the fact that $\mathbb{O}^{\eta}_t$, $t \in [0,T]$, is a $U$-valued Gaussian random variable, we have 
	\begin{align*}
		\mathbb{E} \left[\| \mathbb{O}^{\eta}_t \|_{W^{\beta, p}}^p \right]
		&= \mathbb{E} \left[  \int_{0}^{1} | \mathbb{O}^{\eta}_t(x) |^p dx  +  \int_{0}^{1} \int_{0}^{1} \frac{| \mathbb{O}^{\eta}_t(x) - \mathbb{O}^{\eta}_t(y) |^p}{|x-y|^{1+\beta p}} dx dy \right] \\
		&= \mathbb{E} \left[|Y|^p\right] \int_{0}^{1} \left( \mathbb{E} \left[|\mathbb{O}^{\eta}_t(x) |^2\right] \right)^{\frac{p}{2}} dx  
		+ \mathbb{E} \left[|Y|^p\right] \int_{0}^{1} \int_{0}^{1} \frac{\left(\mathbb{E} \left[| \mathbb{O}^{\eta}_t(x) - \mathbb{O}^{\eta}_t(y) |^2\right]\right)^{\frac{p}{2}}}{|x-y|^{1+\beta p}} dx dy, 
	\end{align*}
	where $Y$ stands for the standard Gaussian  random variable. By the independence of $\{w_k^H\}_{k \in \mathbb{N}}$, It\^o's isometry, and Lemma \ref{Le:S(t)} with $\vartheta = H$, we have 
	\begin{align*}
		\mathbb{E} \left[|\mathbb{O}^{\eta}_t(x) |^2\right] 
		&= \mathbb{E} \left[\left|  \sum_{k=1}^{\infty} \int_{0}^{t} e^{-(\lambda_k+\eta)(t-s)} \phi_k(x) dw_k^H(s) \right|^2\right] \\
		&= \sum_{k=1}^{\infty} |\phi_k(x)|^2 \mathbb{E} \left[\left| \int_{0}^{t} e^{-(\lambda_k+\eta)(t-s)}  dw_k^H(s) \right|^2\right] \\
		&= \sum_{k=1}^{\infty} |\phi_k(x)|^2 \int_{0}^{t} \int_{0}^{t} e^{-(\lambda_k+\eta)(t-u)} e^{-(\lambda_k+\eta)(t-v)} \phi(u-v) du dv \\
		&= \sum_{k=1}^{\infty} \frac{|\phi_k(x)|^2}{(\lambda_k+\eta)^{2H}} \int_{0}^{t} \int_{0}^{t} \left\langle (A+\eta)^H S^{\eta}(t-u) \phi_k ,  (A+\eta)^H S^{\eta}(t-v) \phi_k \right\rangle \phi(u-v) du dv  \\
		&\leq \sum_{k=1}^{\infty} \frac{|\phi_k(x)|^2}{(\lambda_k+\eta)^{2H}} C(H) \|\phi_k\|^2 
		\leq C(H) \sum_{k=1}^{\infty} \frac{1}{(\lambda_k+\eta)^{2H}}, 
	\end{align*}
	which implies that 
	\begin{equation*}
		\int_{0}^{1} \left( \mathbb{E} \left[|\mathbb{O}^{\eta}_t(x) |^2\right] \right)^{\frac{p}{2}} dx \leq (C(H))^{\frac{p}{2}} \left( \sum_{k=1}^{\infty} \frac{1}{(\lambda_k+\eta)^{2H}} \right)^{\frac{p}{2}}. 
	\end{equation*}
	Similarly, we have 
	\begin{equation*}
		\mathbb{E} \left[| \mathbb{O}^{\eta}_t(x) - \mathbb{O}^{\eta}_t(y) |^2 \right]
		\leq C(H) \sum_{k=1}^{\infty} \frac{|\phi_k(x) - \phi_k(y)|^2 }{(\lambda_k+\eta)^{2H}}. 
	\end{equation*}
	Through $|\sin(x) - \sin(y)| \leq |x-y|$ and $|\sin(x)| \leq 1$, then recalling $\phi_k(x) =  \sqrt{2} \sin (k \pi x)$, we have 
	\begin{equation*}
		|\phi_k(x) - \phi_k(y)|^2 = 2 |\sin (k \pi x) - \sin (k \pi y)|^{2-4\beta}  |\sin (k \pi x) - \sin (k \pi y)|^{4\beta} \leq 2^{3-4\beta} (k \pi)^{4\beta} |x-y|^{4\beta}. 
	\end{equation*}
	So, we obtain 
	\begin{equation*}
		\mathbb{E} \left[|\mathbb{O}^{\eta}_t(x) - \mathbb{O}^{\eta}_t(y) |^2\right] 
		\leq C(H) \sum_{k=1}^{\infty} \frac{2^{3-4\beta} (k \pi)^{4\beta} |x-y|^{4\beta}}{(\lambda_k+\eta)^{2H}},  
	\end{equation*}
	which means that 
	\begin{align*}
		&\int_{0}^{1} \int_{0}^{1} \frac{\left(\mathbb{E} \left[| \mathbb{O}^{\eta}_t(x) - \mathbb{O}^{\eta}_t(y) |^2\right]\right)^{\frac{p}{2}}}{|x-y|^{1+\beta p}} dx dy
		\leq  \int_{0}^{1} \int_{0}^{1} \frac{\left( C(H) \sum_{k=1}^{\infty} \frac{2^{3-4\beta} (k \pi)^{4\beta} |x-y|^{4\beta}}{(\lambda_k+\eta)^{2H}} \right)^{\frac{p}{2}}}{|x-y|^{1+\beta p}} dx dy \\
		\leq \ &  (C(H))^{\frac{p}{2}}  2^{\frac{3p}{2} - 2 \beta p} \pi^{2\beta p}  \int_{0}^{1} \int_{0}^{1} \frac{\left( \sum_{k=1}^{\infty} \frac{ k^{4\beta} }{(\lambda_k+\eta)^{2H}} \right)^{\frac{p}{2}}}{|x-y|^{1-\beta p}} dx dy \qquad \qquad (\text{note that } \beta p > 1) \\
		\leq \ &  (C(H))^{\frac{p}{2}}  2^{\frac{3p}{2} - 2 \beta p} \pi^{2\beta p}  \left( \sum_{k=1}^{\infty} \frac{ k^{4\beta} }{(\lambda_k+\eta)^{2H}} \right)^{\frac{p}{2}}. 
	\end{align*}
	These, together with \eqref{Est:Sobolev of O}, lead to  
	\begin{equation*}
		\mathbb{E} \left[\| \mathbb{O}^{\eta}_t \|_{L^\infty}^p\right] 
		\leq C(p, H) \mathbb{E} \left[|Y|^p\right] \left( \sum_{k=1}^{\infty} \frac{1}{(\lambda_k+\eta)^{2H}} \right)^{\frac{p}{2}} + C(p, \beta, H) \mathbb{E} \left[|Y|^p\right] \left( \sum_{k=1}^{\infty} \frac{ k^{4\beta} }{(\lambda_k+\eta)^{2H}} \right)^{\frac{p}{2}}. 
	\end{equation*}
	Finally, through the well-known result $\left(\mathbb{E} \left[|Y|^p\right] \right)^{\frac{1}{p}} \leq p$, we obtain the desired conclusion. 
\end{proof}

We note that $\mathbb{O}^{\eta}_t = \mathcal{O}_t$ when $\eta=0$. The above estimations are thus valid for, e.g. $\mathcal{O}_t$,  $\mathbb{O}^{N, \eta}_t$, and $\mathcal{O}^N_t$. Following Lemma \ref{Le:O eta}, the convergence between $\mathcal{O}_t$ and $\mathcal{O}^N_t$ can be obtained. 
\begin{Corollary}\label{Cor:O - ON}
	For any $p \geq 2$ and $N \in \mathbb{N}$, it holds that  
	\begin{equation*}
		\left( \mathbb{E} \left[\| \mathcal{O}_t - \mathcal{O}^N_t \|_{L^\infty}^p\right] \right)^{\frac{1}{p}}
		\leq C(p, \beta, H) \lambda_{N+1}^{\beta-H+\frac{1}{4}}, 
	\end{equation*}
	where $\beta \in (0, \frac{1}{4})$ and $H \in (\frac{1}{2} , 1)$. 
\end{Corollary}
\begin{proof}
	For $p \in (\frac{1}{\beta}, \infty)$, taking $\eta=0$ in Lemma \ref{Le:O eta} and recalling $\lambda_{k} = k^2 \pi^2$, we obtain 
	\begin{equation*}
		\left( \mathbb{E} \left[\| \mathcal{O}_t - \mathcal{O}^N_t \|_{L^\infty}^p\right] \right)^{\frac{1}{p}}
		\leq C(p, \beta, H) \left( \sum_{k=N+1}^{\infty} \frac{ k^{4\beta} }{\lambda_k^{2H}} \right)^{\frac{1}{2}}
		\leq C(p, \beta, H) \lambda_{N+1}^{\beta-H+\frac{1}{4}}. 
	\end{equation*}
	By using H\"older's inequality, it is easy to yield the desired assertions for $p \in [2, \frac{1}{\beta}]$. 
\end{proof}

\begin{Lemma}\label{Le:eta}
	For $a \in \mathbb{R}$ and $b \in (a+1, \infty)$, it holds that 
		$\lim\limits_{\eta \rightarrow \infty}\left(  \sum_{k=1}^{\infty} \frac{k^a}{k^b+\eta}  \right) = 0$. 
\end{Lemma}
Lemma \ref{Le:eta} can be proved by using Lebesgue's theorem of dominated convergence (see, e.g., \cite{Jentzen2017}). 
Next, we quote Fernique's theorem, which can be found in \cite{Jentzen2017} and \cite[Theorem 8.2.1]{Stroock2011}. It provides the criteria to obtain the boundedness of exponential moments of a stochastic process. 
\begin{Theorem}\label{Th:Fernique}
	Let $(V, \|\cdot\|_V)$ be a separable $\mathbb{R}$-Banach space. Assume that $X : \Omega \rightarrow V$ is a mapping which satisfies that $\forall \, \varphi \in V'$, $\varphi \circ X : \Omega \rightarrow \mathbb{R}$ is a centered Gaussian random variable. If $\mathbb{P}(\|X\|_V^2 > m) \leq \frac{1}{10}$ holds for $m>0$, then 
	\begin{equation*}
		\mathbb{E} \left[ \exp\left( \frac{\|X\|_V^2}{18m} \right) \right] \leq \sqrt{e} + \sum_{k=0}^{\infty} \left( \frac{e}{3} \right)^{2^k} < \infty. 
	\end{equation*}
\end{Theorem}

By using Theorem \ref{Th:Fernique}, we can get the exponential integrability of $\mathbb{O}^{N, \eta}_t$ in the following lemma. 
\begin{Lemma}\label{Le:exponential int O}
	For any $\tilde{c}>0$ and $N \in \mathbb{N}$, there exists $\eta \geq 0$ such that 
	\begin{equation*}
		\mathbb{E}\left[  e^{ \tilde{c} \int_{0}^{T} \left(\| \mathbb{O}^{N, \eta}_t \|_{L^\infty}^2 + 1 \right) dt} \right]
		< \infty.  
	\end{equation*}
\end{Lemma}
\begin{proof}
	For any $\tilde{c}>0$, $\beta \in (0, \frac{1}{4})$ and $H \in (\frac{1}{2} , 1)$, through Lemma \ref{Le:eta}, there exists $\eta$ such that 
	\begin{equation*}
		18 \tilde{c} C(\beta, H) \left( \sum_{k=1}^{N} \frac{ k^{4\beta} }{(\lambda_k+\eta)^{2H}} \right)
	\leq \frac{1}{10}. 
	\end{equation*}
	By Markov's inequality and Lemma \ref{Le:O eta}, we have 
	\begin{equation*}
		\mathbb{P} \left( \| \mathbb{O}^{N, \eta}_t \|_{L^\infty}^2 \geq \frac{1}{18 \tilde{c}} \right) 
		\leq 18 \tilde{c} \ \mathbb{E} \left[\| \mathbb{O}^{N, \eta}_t \|_{L^\infty}^2\right]
		\leq 18 \tilde{c} C(\beta, H) \left( \sum_{k=1}^{N} \frac{ k^{4\beta} }{(\lambda_k+\eta)^{2H}} \right)
		\leq \frac{1}{10}, 
	\end{equation*}
	which satisfies the conditions of Theorem \ref{Th:Fernique}, thus $\mathbb{E} \left[ \exp\left( \tilde{c} \| \mathbb{O}^{N, \eta}_t \|_{L^\infty}^2 \right) \right]  < \infty$. 
	Finally, by Jensen's inequality, we obtain
	\begin{equation*}
		\mathbb{E} \left[ e^{\int_{0}^{T} \tilde{c} \left(\| \mathbb{O}^{N, \eta}_t \|_{L^\infty}^2 + 1 \right) dt} \right] 
		\leq 
		\mathbb{E} \left[ \int_{0}^{T} e^{ \tilde{c} \left(\| \mathbb{O}^{N, \eta}_t \|_{L^\infty}^2 + 1 \right)} dt \right]
		= e^{\tilde{c}} \int_{0}^{T} \mathbb{E} \left[e^{ \tilde{c} \| \mathbb{O}^{N, \eta}_t \|_{L^\infty}^2 }\right] dt
		< \infty, 
	\end{equation*}
	which completes the proof. 
\end{proof}

\subsection{Convergence of the spatial discretization}
Introduce $\tilde{u}(t) := u(t) - \mathcal{O}_t$ and $\tilde{u}^N(t) := u^N(t) - \mathcal{O}^N_t$, which respectively solve 
\begin{align}
	\tilde{u}(t) &= S(t) u_0 + \int_{0}^{t} S(t-s) f(\tilde{u}(s) + \mathcal{O}_s) ds, \label{Eq:mild solution u2} \\ 
	\tilde{u}^N(t) &= S_N(t) P_N u_0 + \int_{0}^{t} S_N(t-s) P_N f(\tilde{u}^N(s) + \mathcal{O}^N_s) ds. \label{Eq:mild solution uN2}
\end{align}
In order to determine the convergence of $\tilde{u}^N(t)$ to $\tilde{u}(t)$ when the discretization parameter $N$ tends to $\infty$, an auxiliary process $P_N \tilde{u}(t)$ is needed, given by 
\begin{equation*}
	P_N \tilde{u}(t) = S_N(t) P_N u_0 + \int_{0}^{t} S_N(t-s) P_N f(\tilde{u}(s) + \mathcal{O}_s) ds. 
\end{equation*}

\begin{Theorem}\label{Th:convergence in probability}
	Under Assumption \ref{Asp:initial value}, for any $a > 0$, it holds that 
	\begin{equation*}
		\lim\limits_{N \rightarrow \infty} \mathbb{P} \left( \sup\limits_{t \in [0,T]}  \| \tilde{u}(t) - \tilde{u}^N(t) \|^2 > a \right) = 0,  
	\end{equation*}
	where $\tilde{u}(t)$ and $\tilde{u}^N(t)$ are given by \eqref{Eq:mild solution u2} and \eqref{Eq:mild solution uN2}, respectively.
\end{Theorem}

\begin{proof}
	For $\alpha \in (0, \frac{1}{4})$, using the triangle inequality and \eqref{projection},  we have  
	\begin{equation}\label{Est:u2 - uN2}
			\| \tilde{u}(t) - \tilde{u}^N(t) \| 
			\leq \| \tilde{u}(t) - P_N \tilde{u}(t) \| + \| P_N \tilde{u}(t) - \tilde{u}^N(t) \| \leq \lambda_{N+1}^{-\frac{\alpha}{2}}  \| \tilde{u}(t) \|_{\dot{H}^{\alpha}} + \|e^N(t) \|, 
	\end{equation}
	in which $e^N(t) := P_N \tilde{u}(t) - \tilde{u}^N(t)$ solves 
	\begin{align*}
		\frac{d}{dt} e^N(t) 
		= -A_N e^N(t) + P_N f(\tilde{u}(t) + \mathcal{O}_t) - P_N f(\tilde{u}^N(t) + \mathcal{O}^N_t), \qquad e^N(0)=0. 
	\end{align*}
	For sake of simplicity, we denote 
	$$g(t) := \tilde{u}(t) + \mathcal{O}_t + \tilde{u}^N(t) + \mathcal{O}^N_t.$$ 
	Then it follows from the identity $\frac{1}{2} \frac{d}{dt} \| e^N(t) \|^2 = \langle e^N(t) , \frac{d}{dt} e^N(t) \rangle$, the integral by parts, and Young's inequality that 
	\begin{align*}
		\frac{1}{2} \frac{d}{dt} \| e^N(t) \|^2 
		&= - \| e^N(t) \|_{\dot{H}^1}^2 - \frac{1}{2} \left\langle \frac{\partial}{\partial x} e^N(t) , (\tilde{u}(t) + \mathcal{O}_t)^2 - (\tilde{u}^N(t) + \mathcal{O}^N_t)^2 \right\rangle  \\
		&\leq - \| e^N(t) \|_{\dot{H}^1}^2 + \frac{1}{4} \| e^N(t) \|_{\dot{H}^1}^2 + \frac{1}{2} \| g(t) \left( \tilde{u}(t) - \tilde{u}^N(t) \right) \|^2 
		+ \frac{1}{2} \| g(t) \left( \mathcal{O}_t - \mathcal{O}^N_t \right) \|^2 \\
		&\leq \frac{1}{2} \| g(t) \|_{L^\infty}^2 \| \tilde{u}(t) - \tilde{u}^N(t) \|^2 + \frac{1}{2} \| g(t) \|^2 \| \mathcal{O}_t - \mathcal{O}^N_t \|_{L^\infty}^2, 
	\end{align*}
	which, together with $e^N(0) = 0$, implies that 
	\begin{equation*}
		\| e^N(t) \|^2
		\leq 
		\int_{0}^{t} \| g(s) \|_{L^\infty}^2 \| \tilde{u}(s) - \tilde{u}^N(s) \|^2 ds 
		+ \int_{0}^{t} \| g(s) \|^2 \| \mathcal{O}_s - \mathcal{O}^N_s \|_{L^\infty}^2 ds. 
	\end{equation*}
	Substituting the above estimate into \eqref{Est:u2 - uN2}, we obtain 
	\begin{equation*}
		\| \tilde{u}(t) - \tilde{u}^N(t) \|^2 
		\leq 2 \lambda_{N+1}^{-\alpha}  \| \tilde{u}(t) \|_{\dot{H}^{\alpha}}^2 
		+ 2 \int_{0}^{t} \| g(s) \|_{L^\infty}^2 \| \tilde{u}(s) - \tilde{u}^N(s) \|^2 ds 
		+ 2 \int_{0}^{t} \| g(s) \|^2 \| \mathcal{O}_s - \mathcal{O}^N_s \|_{L^\infty}^2 ds.  
	\end{equation*}
	Then, it follows from Gronwall's inequality and H\"older's inequality that  
	\begin{equation}\label{Eq:u-uN}
		\begin{aligned}
			\| \tilde{u}(t) - \tilde{u}^N(t) \|^2 
			&\leq \left( 2 \lambda_{N+1}^{-\alpha}  \| \tilde{u}(t) \|_{\dot{H}^{\alpha}}^2 
			+ 2 \int_{0}^{t} \| g(s) \|^2 \| \mathcal{O}_s - \mathcal{O}^N_s \|_{L^\infty}^2 ds \right) e^{2\int_{0}^{t} \| g(s) \|_{L^\infty}^2 ds} \\ 
			&\leq \Bigg[ 2 \lambda_{N+1}^{-\alpha}  \| \tilde{u}(t) \|_{\dot{H}^{\alpha}}^2 
			+ 8 \int_{0}^{t} \left( \| \tilde{u}(s) \|^2 + \| \tilde{u}^N(s) \|^2 \right)\| \mathcal{O}_s - \mathcal{O}^N_s \|_{L^\infty}^2 ds \\
			&\quad + 16 \left( \int_{0}^{t} \| \mathcal{O}_s \|^4 ds \right)^{\frac{1}{2}} \left( \int_{0}^{t} \| \mathcal{O}_s - \mathcal{O}^N_s \|_{L^\infty}^4 ds \right)^{\frac{1}{2}} \Bigg]  e^{2\int_{0}^{t} \| g(s) \|_{L^\infty}^2 ds} . 
		\end{aligned}
	\end{equation}
	
	Next we adopt the stopping time technique to verify the desired conclusion. 
	For $N \in \mathbb{N}$, $\alpha \in (0, \frac{1}{4})$ and a positive real number $\mathcal{R}$, we introduce a sequence of stopping times as follows: 
	\begin{align*}
		\tau_1 &:= T \wedge \inf \{ t \geq 0: \int_{0}^{t} \|\tilde{u}(s)\|_{L^\infty}^2 ds \geq \mathcal{R} \}, \\
		\tau_2 &:= T \wedge \inf \{ t \geq 0: \|\tilde{u}(t)\|^2 \geq \mathcal{R} \}, \\
		\tau_3 &:= T \wedge \inf \{ t \geq 0: \int_{0}^{t} \|\mathcal{O}_s\|_{L^\infty}^{\frac{14+4\alpha}{1-2\alpha}} ds \geq \mathcal{R} \}, \\
		\tau^N_1 &:= T \wedge \inf \{ t \geq 0: \int_{0}^{t} \|\tilde{u}^N(s)\|_{L^\infty}^2 ds \geq \mathcal{R} \}, \\
		\tau^N_2 &:= T \wedge \inf \{ t \geq 0: \|\tilde{u}^N(t)\|^2 \geq \mathcal{R} \},\\
		\tau^N_3 &:= T \wedge \inf \{ t \geq 0: \int_{0}^{t} \|\mathcal{O}^N_s\|_{L^\infty}^2 ds \geq \mathcal{R} \},
	\end{align*}
	where we set $\inf \varnothing = + \infty$. 
	Denote 
	$$\tau := \tau_{N,\mathcal{R}}(\omega) := \min \left\{ \tau_1, \tau_2, \tau_3, \tau^N_1, \tau^N_2, \tau_3^N \right\}. $$
	Obviously, the stopping time $\tau: \Omega \rightarrow [0,T]$ depends on parameters $N$ and $\mathcal{R}$. 
	For $\tilde{u}(t)$ given by \eqref{Eq:mild solution u2}, taking $\zeta=\alpha$ in \eqref{Est:S B 1}, then using H\"older's inequality, we get 
	\begin{align*}
		\|\tilde{u}(t)\|_{\dot{H}^\alpha} 
		&\leq \|S(t) u_0\|_{\dot{H}^\alpha} + \frac{1}{2} \int_{0}^{t} \left\| S(t-s) \frac{\partial}{\partial x} \left( \tilde{u}(s) + \mathcal{O}_s \right)^2 \right\|_{\dot{H}^\alpha} ds \\
		&\leq \|u_0\|_{\dot{H}^\alpha} + C(\alpha) \int_{0}^{t} \left(t-s\right)^{-\frac{\alpha}{2}-\frac{3}{4}}  \left\| \frac{\partial}{\partial x} \left( \tilde{u}(s) + \mathcal{O}_s \right)^2 \right\|_{W^{-1 , 1}} ds \\ 
		&\leq \|u_0\|_{\dot{H}^\alpha} + 2 C(\alpha) \int_{0}^{t}  \left(t-s\right)^{-\frac{\alpha}{2}-\frac{3}{4}} \| \tilde{u}(s) \|^2 ds
		+ 2 C(\alpha) \int_{0}^{t}  \left(t-s\right)^{-\frac{\alpha}{2}-\frac{3}{4}} \| \mathcal{O}_s \|^2 ds \\
		&\leq \|u_0\|_{\dot{H}^\alpha} + 2C(\alpha) \int_{0}^{t}  \left(t-s\right)^{-\frac{\alpha}{2}-\frac{3}{4}} \| \tilde{u}(s) \|^2 ds \\
		&\quad + 2C(\alpha) \left( \int_{0}^{t}  \left(t-s\right)^{-\frac{7+2\alpha}{8} } ds \right)^{\frac{6+4\alpha}{7+2\alpha}} \left( \int_{0}^{t} \| \mathcal{O}_s \|^{\frac{14+4\alpha}{1-2\alpha}} ds \right)^{\frac{1-2\alpha}{7+2\alpha}}. 
	\end{align*}
	In view of the above estimate and the definition of $\tau$, for $\alpha \in (0, \frac{1}{4})$, we have 
	\begin{equation}\label{Eq:u alpha}
		\begin{aligned}
			\sup\limits_{t \leq \tau} \| \tilde{u}(t) \|_{\dot{H}^{\alpha}}^2 
			&\leq \sup\limits_{t \leq \tau} \Bigg\{ 3 \|u_0\|_{\dot{H}^\alpha}^2 + 12 (C(\alpha))^2 \left( \int_{0}^{t}  \left(t-s\right)^{-\frac{\alpha}{2}-\frac{3}{4}  } \| \tilde{u}(s) \|^2 ds \right)^2 \\
			&\quad + 12 (C(\alpha))^2 \left( \int_{0}^{t}  \left(t-s\right)^{-\frac{7+2\alpha}{8} } ds \right)^{\frac{12+8\alpha}{7+2\alpha}} \left( \int_{0}^{t} \| \mathcal{O}_s \|^{\frac{14+4\alpha}{1-2\alpha}} ds \right)^{\frac{2-4\alpha}{7+2\alpha}}  \Bigg\} \\
			&\leq 3 \|u_0\|_{\dot{H}^\alpha}^2 + C(\mathcal{R}, T, \alpha). 
		\end{aligned}
	\end{equation}
	Hence, from \eqref{Eq:u-uN} and \eqref{Eq:u alpha}, then by the definition of $\tau$ and H\"older's inequality, we have 
	\begin{equation}\label{Est:u2 - uN2 2}
		\sup\limits_{t \leq \tau}  \| \tilde{u}(t) - \tilde{u}^N(t) \|^2  
		\leq C(\mathcal{R}, T, \alpha) \left( \lambda_{N+1}^{-\alpha} \left(\| u_0 \|_{\dot{H}^{\alpha}}^2 + 1\right)
		+ \left( \int_{0}^{T} \| \mathcal{O}_s - \mathcal{O}^N_s \|_{L^\infty}^4 ds \right)^{\frac{1}{2}} \right). 
	\end{equation}
	We claim that  
	\begin{equation}\label{tau<T}
		\lim\limits_{\mathcal{R} \rightarrow \infty} \mathbb{P} (\tau < T) = 0 \quad  \text{holds for any } N \in \mathbb{N} . 
	\end{equation}
	Obviously, it suffices to prove \eqref{tau<T} for each stopping times $\tau_i$ and $\tau_i^N$, $i = 1, 2, 3$.

	For $\tau_1$, by $\dot{H}^1 \hookrightarrow L^\infty(\mathcal{D})$, and taking $\eta = 0$ in the estimates \eqref{Est:u eta} and \eqref{Est:int u eta}, we have 
	\begin{align*} 
		&\frac{1}{2} \int_{0}^{t} \| \tilde{u}(s) \|_{L^\infty}^2 ds 
		\leq \frac{1}{2} \int_{0}^{t} \| \tilde{u}(s) \|_{\dot{H}^1}^2 ds \\
		\leq \ & \|u_0\|^2 + \int_{0}^{t} \left(\| \mathcal{O}_s \|_{L^\infty}^2 + 1 \right)  \| \tilde{u}(s)\|^2  ds
		+ \frac{1}{2} \int_{0}^{t} \| \mathcal{O}_s \|_{L^4}^4 ds \\
		\leq \ &  \|u_0\|^2
		+ \frac{1}{2} \int_{0}^{t} \| \mathcal{O}_s \|_{L^4}^4 ds 
		+ \int_{0}^{t} \left(\| \mathcal{O}_s \|_{L^\infty}^2 + 1 \right) \left( \|u_0\|^2 + \frac{1}{2} \int_{0}^{s} \| \mathcal{O}_r \|_{L^4}^4 dr \right)
		e^{\int_{0}^{s} \left(\| \mathcal{O}_r \|_{L^\infty}^2 + 1 \right) dr}  ds \\
		\leq \ &  \|u_0\|^2
		+ \frac{1}{2} \int_{0}^{t} \| \mathcal{O}_s \|_{L^4}^4 ds 
		+ \left(\int_{0}^{t} \left(\| \mathcal{O}_s \|_{L^\infty}^2 + 1 \right) ds \right) \left( \|u_0\|^2 + \frac{1}{2} \int_{0}^{t} \| \mathcal{O}_s \|_{L^4}^4 ds \right)
		e^{\int_{0}^{t} \left(\| \mathcal{O}_s \|_{L^\infty}^2 + 1 \right) ds} \\
		\leq \ &  \Bigg[ \|u_0\|^2
		+ \frac{1}{2} \left( \int_{0}^{t} \| \mathcal{O}_s \|_{L^4}^4 ds \right) \left( \int_{0}^{t} \| \mathcal{O}_s \|_{L^\infty}^2 ds + 1 + t \right) \\ 
		& + \|u_0\|^2 \int_{0}^{t} \left(\| \mathcal{O}_s \|_{L^\infty}^2 + 1 \right) ds  \Bigg]
		e^{\int_{0}^{t} \left(\| \mathcal{O}_s \|_{L^\infty}^2 + 1 \right) ds}. 
	\end{align*}
	It follows from Jensen's inequality, $ab \leq \frac{1}{2} a^2 + \frac{1}{2} b^2$, and H\"older's inequality that  
	\begin{align*} 
		&\mathbb{E} \left[ \log \left( \dfrac{1}{2} \int_{0}^{t} \| \tilde{u}(s) \|_{L^\infty}^2 ds \right) \right] \\
		\leq \ &  \mathbb{E} \Bigg[ \log \Bigg\{ \|u_0\|^2
		+ \frac{1}{2} \left( \int_{0}^{t} \| \mathcal{O}_s \|_{L^4}^4 ds \right) \left( \int_{0}^{t} \| \mathcal{O}_s \|_{L^\infty}^2 ds + 1 + t \right) \\
		& + \|u_0\|^2 \int_{0}^{t} \left(\| \mathcal{O}_s \|_{L^\infty}^2 + 1 \right) ds  \Bigg\} + \int_{0}^{t} \left(\| \mathcal{O}_s \|_{L^\infty}^2 + 1 \right) ds \Bigg] \\
		\leq \ &  \log \Bigg\{ \mathbb{E} \left[\|u_0\|^2\right]
		+ \frac{1}{2} \mathbb{E} \left[ \left( \int_{0}^{t} \| \mathcal{O}_s \|_{L^4}^4 ds \right) \left( \int_{0}^{t} \| \mathcal{O}_s \|_{L^\infty}^2 ds + 1 + t \right) \right] \\
		& + \mathbb{E} \left[ \|u_0\|^2 \int_{0}^{t} \left(\| \mathcal{O}_s \|_{L^\infty}^2 + 1 \right) ds \right] \Bigg\} + \int_{0}^{t} \mathbb{E} \left[\| \mathcal{O}_s \|_{L^\infty}^2 + 1 \right] ds \\
		\leq \ &  \log \Bigg\{ \mathbb{E} \left[\|u_0\|^2\right]
		+ \frac{t}{4} \int_{0}^{t} \mathbb{E} \left[\| \mathcal{O}_s \|_{L^4}^8\right] ds  + \frac{3}{4}  \left( t \int_{0}^{t} \mathbb{E} \left[\|\mathcal{O}_s \|_{L^\infty}^4\right] ds + 1 + t^2 \right)  \\ 
		& + \frac{1}{2} \mathbb{E} \left[\|u_0\|^4\right] + t \int_{0}^{t} \mathbb{E} \left[\|\mathcal{O}_s \|_{L^\infty}^2\right] ds + t^2 \Bigg\} + \int_{0}^{t} \mathbb{E} \left[\| \mathcal{O}_s \|_{L^\infty}^2\right] ds + t. 
	\end{align*}
	By Markov's inequality and the above estimate, we obtain 
	\begin{align*}
		& \mathbb{P}(\tau_1 < T) 
		= \mathbb{P}\left( \log \Big( \frac{1}{2} \int_{0}^{T} \|\tilde{u}(s)\|_{L^\infty}^2 ds \Big) \geq \log \left( \tfrac{\mathcal{R}}{2} \right) \right) \\
		\leq \ &  \frac{1}{\ \log \left( \tfrac{\mathcal{R}}{2} \right) \ } \log \Bigg\{ \mathbb{E} \left[\|u_0\|^2\right]
		+ \frac{T}{4} \int_{0}^{T} \mathbb{E} \left[\| \mathcal{O}_s \|_{L^4}^8\right] ds  + \frac{3T}{4}  \int_{0}^{T} \mathbb{E} \left[\| \mathcal{O}_s \|_{L^\infty}^4\right] ds + \frac{3}{4} \left( 1 + T^2 \right)  \\ 
		& + \frac{1}{2} \mathbb{E} \left[\|u_0\|^4\right] + T \int_{0}^{T} \mathbb{E} \left[\| \mathcal{O}_s \|_{L^\infty}^2\right] ds + T^2 \Bigg\} + \frac{1}{\ \log \left( \tfrac{\mathcal{R}}{2} \right) \ } \left( \int_{0}^{T} \mathbb{E} \left[\| \mathcal{O}_s \|_{L^\infty}^2\right] ds + T \right). 
	\end{align*}
	This verifies $\mathbb{P}(\tau_1 < T) \rightarrow 0$ when $\mathcal{R} \rightarrow \infty$ due to the boundedness of moments of the stochastic convolution shown in Lemma \ref{Le:O eta} with $\eta = 0$. Similar estimations can be applied to derive that: for any $N \in \mathbb{N}$, $\mathbb{P}(\tau_1^N < T) \rightarrow 0$ when $\mathcal{R} \rightarrow \infty$. 
	
	For $\tau_2$, since the characteristic function $\chi_{\{\tau_2 < T\}}\leq1$ and
	$$\log \left(\sup_{t\in[0,T]}\|\tilde{u}(t)\|^2\right) \leq \log \left(\sup_{t\in[0,T]}\|\tilde{u}(t)\|^2+1\right),$$
	we can apply \eqref{Est:u eta} with $\eta=0$ and Jensen's inequality to obtain
	\begin{align*}
		\mathbb{P}(\tau_2 < T) 
		&= \mathbb{E}\left[ \chi_{\{\tau_2 < T\}} \right]
		\leq \mathbb{E}\left[ \chi_{\{\tau_2 < T\}} \frac{\log \left(\|\tilde{u}(\tau_2)\|^2\right)}{\log \mathcal{R}} \right]
		\leq \mathbb{E}\left[ \chi_{\{\tau_2 < T\}} \frac{ \log \left(\sup_{t \in [0,T]} \|\tilde{u}(t)\|^2 \right)}{\log \mathcal{R}} \right] \\
		& \leq \mathbb{E}\left[ \frac{ \log \left(\sup_{t \in [0,T]} \|\tilde{u}(t)\|^2 + 1 \right)}{\log \mathcal{R}} \right] \\
		&\leq \frac{1}{\log \mathcal{R}} \mathbb{E}\left[ \log \left\{ \left( \|u_0\|^2 + \frac{1}{2} \int_{0}^{T} \| \mathcal{O}_t \|_{L^4}^4 dt + 1 \right)
		e^{\int_{0}^{T} \left(\| \mathcal{O}_t \|_{L^\infty}^2 + 1 \right) dt} \right\} \right] \\
		&\leq \frac{1}{\log \mathcal{R}} \left( \log \left\{ \mathbb{E} \left[\|u_0\|^2\right] + \frac{1}{2} \int_{0}^{T} \mathbb{E} \left[\| \mathcal{O}_t \|_{L^4}^4\right] dt + 1 \right\} + \int_{0}^{T} \left( \mathbb{E} \left[\| \mathcal{O}_t \|_{L^\infty}^2\right] + 1 \right) dt \right), 
	\end{align*}
	which, combined with Lemma \ref{Le:O eta} with $\eta = 0$, implies that $\mathbb{P}(\tau_2 < T) \rightarrow 0$ when $\mathcal{R} \rightarrow \infty$. 
	A similar result for $\tau^N_2$ can be proved analogously.

	For $\tau_3$, since $t \mapsto \int_{0}^{t} \|\mathcal{O}_s\|_{L^\infty}^p ds$ is continuous and monotonously increasing, through Markov's inequality and taking $p = \frac{14+4\alpha}{1-2\alpha}$, we have 
	\begin{equation*}
		\mathbb{P}(\tau_3 < T) = \mathbb{P} \left( \int_{0}^{T} \|\mathcal{O}_s\|_{L^\infty}^p ds > \mathcal{R} \right) \leq \frac{\mathbb{E} \left[\int_{0}^{T} \|\mathcal{O}_s\|_{L^\infty}^p ds\right]}{\mathcal{R}} \xrightarrow{\mathcal{R} \rightarrow \infty} 0, 
	\end{equation*} 
	in which the boundedness of $\sup_{t \in [0,T]}\mathbb{E} \left[\|\mathcal{O}_t\|_{L^\infty}^p\right]$ proved in Lemma \ref{Le:O eta} with $\eta = 0$ is used. 
	The case of $\tau^N_3$ is the analogous consequence. Consequently, we verify the claim \eqref{tau<T}. 
	
	For any $a>0$, we have 
	\begin{align*}
		&\mathbb{P} \left( \sup\limits_{t \in [0,T]}  \| \tilde{u}(t) - \tilde{u}^N(t) \|^2 > a \right) \\
		= \ & \mathbb{P} \left( \sup\limits_{t \in [0,T]}  \| \tilde{u}(t) - \tilde{u}^N(t) \|^2 > a \ \text{and} \ \tau = T \right) 
		+ \mathbb{P} \left( \sup\limits_{t \in [0,T]}  \| \tilde{u}(t) - \tilde{u}^N(t) \|^2 > a \ \text{and} \ \tau < T \right) \\
		\leq \ & \mathbb{P} \left( \sup\limits_{t \in [0,T]} \| \tilde{u}(t) - \tilde{u}^N(t) \|^2 > a \ \text{and} \ \tau = T \right) 
		+ \mathbb{P} \left( \tau < T \right). 
	\end{align*}
	For the first probability, replacing $\tau$ by $T$ in \eqref{Est:u2 - uN2 2}, then by Markov's inequality and Corollary \ref{Cor:O - ON}, we obtain 
	\begin{align*}
		&\mathbb{P} \left( \sup\limits_{t \in [0,T]}  \| \tilde{u}(t) - \tilde{u}^N(t) \|^2 > a \ \text{and} \ \tau = T \right) \\ 
		\leq \ & \frac{C(\mathcal{R}, T, \alpha)}{a} \left[ \lambda_{N+1}^{-\alpha} \left( \mathbb{E}\left[\|u_0\|_{\dot{H}^{\alpha}}^2\right] + 1\right)
		+ \left( \int_{0}^{T} \mathbb{E}\left[\| \mathcal{O}_s - \mathcal{O}^N_s \|_{L^\infty}^4\right] ds \right)^{\frac{1}{2}}\right] \\ 
		\leq \ & \frac{C(\mathcal{R}, H,T,\alpha)}{a} \left( \lambda_{N+1}^{-\alpha}
		+ \lambda_{N+1}^{2\alpha-2H+\frac{1}{2}} \right), 
	\end{align*}
	where $\alpha \in (0, \frac{1}{4})$ and $H \in (\frac{1}{2}, 1)$. 
	From \eqref{tau<T}, for any $\epsilon > 0$ and $N \in \mathbb{N}$, there exists $\mathcal{R}$ independent of $N$ such that 
	\begin{equation*}
		\mathbb{P} \left( \tau < T \right) < \frac{\epsilon}{2}. 
	\end{equation*}
	For the fixed $\mathcal{R}$ and for any $\epsilon>0$, there exists $N \in \mathbb{N}$ depending on $\epsilon$ and $\mathcal{R}$ such that 
	\begin{equation*}
		\mathbb{P} \left( \sup\limits_{t \in [0,T]}  \| \tilde{u}(t) - \tilde{u}^N(t) \|^2 > a \ \text{and} \ \tau = T \right) < \frac{\epsilon}{2}. 
	\end{equation*}
	Thus, the theorem is completely proved. 
\end{proof}

The following lemma means that the convergence in probability and the boundedness of moments of numerical solutions are enough to get the convergence in the strong sense. 
\begin{Lemma}[Proposition 4.1 in \cite{Hutzenthaler2018}]\label{Le:probability to expectation}
	Suppose that $I$ is a non-empty set and $U$ is a separable normed $\mathbb{R}$-vector space. 
	For $X^n: I \times \Omega \rightarrow U$, $n \in \mathbb{N} \cup \{0\}$ satisfying the following conditions, 
	\begin{align*}
		\limsup_{n \rightarrow \infty} \sup_{i \in I} \mathbb{P} \left( \|X_i^n-X_i^0\|_{U} \geq \epsilon \right) = 0, \quad \forall \, \epsilon > 0, \\	
		\limsup_{n \rightarrow \infty} \sup_{i \in I} \mathbb{E} \left[\|X_i^n\|_{U}^p\right] < \infty, \quad p \in (0,\infty), 
	\end{align*}
	it holds that for any $q \in (0,p)$, 
	\begin{align*}
		\limsup_{n \rightarrow \infty} \sup_{i \in I} \mathbb{E} \left[\|X_i^n - X_i^0\|_{U}^q\right] = 0 \quad  \text{and} \quad \sup_{i \in I} \mathbb{E} \left[\|X_i^0\|_{U}^q\right] < \infty. 
	\end{align*}
\end{Lemma}
By virtue of above preparations, the strong convergence of the spatial semi-discretization is presented as follows.  
\begin{Theorem}\label{thm:convergence of  semi-discrete}
	Under Assumption \ref{Asp:initial value}, it holds that for $p>0$, 
	\begin{equation*}
		\sup\limits_{t \in [0,T]} \mathbb{E} \left[\| u(t) - u^N(t) \|^p\right] \rightarrow 0, \quad \text{as } N \rightarrow \infty, 
	\end{equation*} 
	where $u(t)$ and $u^N(t)$ are given by \eqref{Eq:mild solution u1} and \eqref{Eq:mild solution uN}, respectively. 
\end{Theorem}

\begin{proof}
	It suffices to verify the conditions of Lemma \ref{Le:probability to expectation}. 
	On the one hand, by Theorem \ref{Th:convergence in probability}, 
	\begin{equation*}
		\lim\limits_{N \rightarrow \infty} \mathbb{P} 	\left( \sup\limits_{t \in [0,T]}  \| \tilde{u}(t) - \tilde{u}^N(t) \|^p > a \right) = 0, \quad \forall \, a > 0. 
	\end{equation*}
	On the other hand, it follows from \eqref{Est:u eta N}, H\"older's inequality, Lemma \ref{Le:exponential int O}, and Lemma \ref{Le:O eta} that 
	\begin{equation*}
		\lim\limits_{N \rightarrow \infty} \mathbb{E} \left[\| \tilde{u}^N_\eta(t) \|^p\right] < \infty. 
	\end{equation*}
	In addition, recalling $\tilde{u}^N_{\eta}(t) = u^N(t) - \mathbb{O}^{N, \eta}_t$ and $\tilde{u}^N(t) = u^N(t) - \mathcal{O}^N_t$, which means that $\tilde{u}^N(t) = \tilde{u}^N_{\eta}(t) + \mathbb{O}^{N, \eta}_t - \mathcal{O}^N_t$, and thus 
	\begin{equation*}
		\lim\limits_{N \rightarrow \infty} \mathbb{E} \left[\| \tilde{u}^N(t) \|^p \right]
		\leq \lim\limits_{N \rightarrow \infty} 3^{p-1} \left( \mathbb{E} \left[\|\tilde{u}^N_{\eta}(t) \|^p\right]
		+ \mathbb{E} \left[\|\mathbb{O}^{N, \eta}_t\|^p\right]
		+ \mathbb{E} \left[\|\mathcal{O}^N_t\|^p\right] \right)
		< \infty. 
	\end{equation*} 
	Consequently, through Lemma \ref{Le:probability to expectation}, we obtain 
	\begin{equation*}
		\sup\limits_{t \in [0,T]} \mathbb{E} \left[\|\tilde{u}(t) - \tilde{u}^N(t)\|^p\right] \rightarrow 0, \quad \text{as } N \rightarrow \infty. 
	\end{equation*}
	Finally, by using Corollary \ref{Cor:O - ON}, we immediately obtain the following desired conclusion: 
	\begin{align*}
		&\sup\limits_{t \in [0,T]} \mathbb{E} \left[\|u(t) - u^N(t)\|^p\right]  \\
		= \ & \sup\limits_{t \in [0,T]} \mathbb{E}  \left[\| \tilde{u}(t) + \mathcal{O}_t - \tilde{u}^N(t) - \mathcal{O}^N_t \|^p\right] \\
		\leq \ & 2^{p-1} \sup\limits_{t \in [0,T]} \mathbb{E}  \left[\| \tilde{u}(t) - \tilde{u}^N(t)\|^p\right] + 2^{p-1} \sup\limits_{t \in [0,T]} \mathbb{E} \left[\| \mathcal{O}_t - \mathcal{O}^N_t \|^p\right]  
		\rightarrow 0, \quad \text{as } N \rightarrow \infty, 
	\end{align*}
	which completes the proof. 
\end{proof}

\section{Full discretization}\label{Sec:full-discrete}
In this section we construct a fully space-time discrete scheme to numerically solve the SBE at $t_m$, $m = 0, 1, ..., M$, $M \in \mathbb{N}$, where $t_m = m \tau$ are gridpoints of a uniform partition of $[0,T]$ and $\tau = \frac{T}{M}$ is the temporal step size. 
Based on the spatial discretization presented in Section \ref{Sec:semi-discrete}, we introduce a tamed accelerated exponential Euler scheme, given by $v_{t_0} = P_N u_0$ and for $m = 0, 1, ..., M-1$, 
\begin{equation*}
	v_{t_{m+1}} = S_N(\tau) v_{t_{m}} + \int_{t_{m}}^{t_{m+1}} \frac{S_N(t_{m+1}-s) P_N f(v_{t_{m}})}{1+\tau^{\theta}\|v_{t_{m}}\|_{\dot{H}^{\varrho}} ^2 + \tau^{\theta}\|\mathcal{O}^N_{t_m}\|_{\dot{H}^{\varrho}}^2 } ds 
	+ \int_{t_{m}}^{t_{m+1}} S_N(t_{m+1}-s) P_N d B^H(s), 
\end{equation*}
in which by $v_{t_{m}}$, $m = 0, 1, ..., M$, we denote the full-discrete numerical solutions. 
In addition, we set 
\begin{equation}\label{rho and theta}
	\frac{1}{4} \leq \varrho \leq \frac{3}{8}, \quad  \varrho - 2\theta \geq \frac{1}{4}, \quad \theta < \frac{1}{8}. 
\end{equation}
These settings are attainable, for example, taking $\varrho = \frac{3}{8}$ and $\theta = \frac{1}{16}$. 
Define $\kappa(t) := \tau \left\lfloor \frac{t}{\tau} \right\rfloor$ for $t \in [0,T]$, where $\lfloor \cdot \rfloor$ represents the floor function. Namely, 
\begin{equation*}
	\kappa(t) = t_m = m \tau, \ \ \text{for} \ t \in [t_m, t_{m+1}), \ m = 0, 1, ... , M-1.
\end{equation*} 
With $\kappa(t)$ defined above, we introduce the following continuous-time approximations:  
\begin{equation}\label{Eq:continuous version}
	v_t = S_N(t) v_{t_{0}} + \int_{0}^{t} \frac{S_N(t-s) P_N f(v_{\kappa(s)})}{1+\tau^{\theta}\|v_{\kappa(s)}\|_{\dot{H}^{\varrho}}^2 + \tau^{\theta}\|\mathcal{O}^N_{\kappa(s)}\|_{\dot{H}^{\varrho}} ^2 } ds + \mathcal{O}^N_{t}. 
\end{equation}
We note that the values of $v_t$ at gridpoints $t_m$ coincide with full-discrete numerical solutions $v_{t_m}$. 
In fact, substituting $t=t_{m+1}$ into \eqref{Eq:continuous version}, we obtain 
\begin{align*}
	v_{t_{m+1}} &= S_N(t_{m+1}) v_{t_{0}} + \int_{0}^{t_{m+1}} \frac{S_N(t_{m+1}-s) P_N f(v_{\kappa(s)})}{1+\tau^{\theta}\|v_{\kappa(s)}\|_{\dot{H}^{\varrho}}^2 + \tau^{\theta}\|\mathcal{O}^N_{\kappa(s)}\|_{\dot{H}^{\varrho}} ^2 } ds + \int_{0}^{t_{m+1}} S_N(t_{m+1}-s) P_N dB^H(s) \\
	&=S_N(\tau)\left( S_N(t_{m}) v_{t_{0}} +  \int_{0}^{t_{m}} \frac{S_N(t_{m}-s) P_N f(v_{\kappa(s)})}{1+\tau^{\theta}\|v_{\kappa(s)}\|_{\dot{H}^{\varrho}}^2 + \tau^{\theta}\|\mathcal{O}^N_{\kappa(s)}\|_{\dot{H}^{\varrho}} ^2 } ds +  \int_{0}^{t_{m}} S_N(t_{m}-s) P_N dB^H(s)\right) \\[2mm]
	&\quad + \int_{t_m}^{t_{m+1}} \frac{S_N(t_{m+1}-s) P_N f(v_{t_m})}{1+\tau^{\theta}\|v_{t_m}\|_{\dot{H}^{\varrho}}^2 + \tau^{\theta}\|\mathcal{O}^N_{t_m}\|_{\dot{H}^{\varrho}} ^2 } ds + \int_{t_m}^{t_{m+1}} S_N(t_{m+1}-s) P_N dB^H(s) \\
	&= S_N(\tau) v_{t_{m}} + \int_{t_{m}}^{t_{m+1}} \frac{S_N(t_{m+1}-s) P_N f(v_{t_{m}})}{1+\tau^{\theta}\|v_{t_{m}}\|_{\dot{H}^{\varrho}} ^2 + \tau^{\theta}\|\mathcal{O}^N_{t_m}\|_{\dot{H}^{\varrho}}^2 } ds 
	+ \int_{t_{m}}^{t_{m+1}} S_N(t_{m+1}-s) P_N d B^H(s), 
\end{align*}
which exactly matches the numerical scheme.

\subsection{Pathwise estimate of fully discrete approximations}
In the remainder of this section, we always assume that $\varrho$ and $\theta$ satisfy \eqref{rho and theta}. 
Denote $\tilde{v}_t := v_t - \mathbb{O}^{N,\eta}_{t}$, then 
\begin{equation}\label{Eq:continuous differential version}
	\frac{d}{dt} \tilde{v}_t = -A_N \tilde{v}_t + \frac{P_N f(\tilde{v}_{\kappa(t)}+\mathbb{O}^{N,\eta}_{\kappa(t)})}{1+\tau^{\theta}\|v_{\kappa(t)}\|_{\dot{H}^{\varrho}}^2 + \tau^{\theta}\|\mathcal{O}^N_{\kappa(t)}\|_{\dot{H}^{\varrho}} ^2 } + \eta \mathbb{O}^{N,\eta}_t, \qquad \tilde{v}_0=0.  
\end{equation}
For convenience, we introduce a stochastic process $G(t): [0,T] \times \Omega \rightarrow (0, 1)$, given by
$$G(t) := \frac{1}{1+\tau^{\theta}\|v_{\kappa(t)}\|_{\dot{H}^{\varrho}} ^2 + \tau^{\theta}\|\mathcal{O}^N_{\kappa(t)}\|_{\dot{H}^{\varrho}} ^2 }. $$
It can be seen that $G(t)$ is a step function, i.e., $G(t) = G(\kappa(t))$ for any $t \in [0,T]$. 
By using $\frac{1}{2} \frac{d}{dt} \|\tilde{v}_t\|^2 = \langle \tilde{v}_t , \frac{d}{dt} \tilde{v}_t \rangle$, the integration by parts, Young's inequality, and H\"older's inequality, we obtain from \eqref{Eq:continuous differential version} that 
\begin{equation}\label{Eq:derivetive of vbar}
	\begin{aligned} 
		&\frac{1}{2} \frac{d}{dt} \|\tilde{v}_t\|^2 
		= - \| \tilde{v}_t \|_{\dot{H}^1}^2 - \frac{G(t)}{2} \left\langle \frac{\partial}{\partial x} \tilde{v}_t , (v_{\kappa(t)})^2 - (v_t)^2 \right\rangle 
		- \frac{G(t)}{2} \left\langle \frac{\partial}{\partial x} \tilde{v}_t , (v_t)^2 \right\rangle 
		+ \eta \langle \tilde{v}_t , \mathbb{O}^{N,\eta}_t \rangle \\
		&\leq -\frac{1}{4} \| \tilde{v}_t \|_{\dot{H}^1}^2 + \frac{(G(t))^2}{4} \left\| v_t + v_{\kappa(t)} \right\|^2_{L^4} \left\| v_t - v_{\kappa(t)} \right\|^2_{L^4}
		+ 2\|\mathbb{O}^{N,\eta}_t\|_{L^\infty}^2\|\tilde{v}_t\|^2 + \frac{1}{2} \|\mathbb{O}^{N,\eta}_t\|_{L^4}^4 + \eta^2 \|\mathbb{O}^{N,\eta}_t\|^2  . 
	\end{aligned}
\end{equation}
On the one hand, by denoting $\bar{v}_t := v_t - \mathcal{O}^N_t$, we know that 
\begin{equation*}
	\bar{v}_t = S_N(t-\kappa(t)) \bar{v}_{\kappa(t)} + \int_{\kappa(t)}^{t} G(s) S_N(t-s) P_N f(v_{\kappa(s)}) ds, 
\end{equation*}
which, together with Sobolev embedding inequalities $\dot{H}^{\varrho} \hookrightarrow L^4(\mathcal{D})$ and $\dot{H}^{1-2\varrho} \hookrightarrow L^4(\mathcal{D})$, \eqref{Est:S B 2} with $\zeta=1-2\varrho$, and $\bar{v}_t = v_t - \mathcal{O}^N_{t}$, implies that 
\begin{equation}\label{Eq:vbar L4}
	\begin{aligned}
		\|\bar{v}_t\|_{L^4} 
		&\leq C \| S_N(t-\kappa(t)) \bar{v}_{\kappa(t)}\|_{\dot{H}^{\varrho}} + \frac{CG(s)}{2}  \int_{\kappa(t)}^{t} \left\| S_N(t-s) \frac{\partial}{\partial x} (v_{\kappa(t)})^2 \right\|_{\dot{H}^{1-2\varrho}} ds \\
		&\leq C\|\bar{v}_{\kappa(t)}\|_{\dot{H}^{\varrho}} +  \frac{C(\varrho) G(\kappa(t))}{2} \int_{\kappa(t)}^{t} (t-s)^{\varrho-1} \left\| \frac{\partial}{\partial x} (v_{\kappa(t)})^2 \right\|_{W^{-1, 2}} ds \\
		&\leq C \| v_{\kappa(t)}\|_{\dot{H}^{\varrho}} + C \|\mathcal{O}^N_{\kappa(t)}\|_{\dot{H}^{\varrho}} + C(\varrho)  G(\kappa(t)) \tau^{\varrho} \| v_{\kappa(t)} \|_{L^4}^2. 
	\end{aligned}
\end{equation}
On the other hand, we know that  
\begin{equation*}
	\bar{v}_t - \bar{v}_{\kappa(t)} = \left(S_N(t-\kappa(t)) - I\right) \bar{v}_{\kappa(t)} + \int_{\kappa(t)}^{t} G(s) S_N(t-s) P_N f(v_{\kappa(s)}) ds. 
\end{equation*}
Similarly to the estimation of \eqref{Eq:vbar L4}, by Sobolev embedding inequalities $\dot{H}^{\varrho-2\theta} \hookrightarrow L^4(\mathcal{D})$ and $\dot{H}^{1-2\varrho} \hookrightarrow L^4(\mathcal{D})$, then taking $\zeta=1-2\varrho$ in \eqref{Est:S B 2}, we obtain
\begin{equation}\label{Eq:vbar-vbark L4}
	\begin{aligned}
		&\|\bar{v}_t - \bar{v}_{\kappa(t)}\|_{L^4} \\
		\leq \ &  C \left\| \left(S_N(t-\kappa(t))-I\right) A^{-\theta} A^{\frac{\varrho}{2}} \bar{v}_{\kappa(t)} \right\| 
		+ \frac{CG(\kappa(t))}{2} \int_{\kappa(t)}^{t} \left\| S_N(t-s) \frac{\partial}{\partial x} (v_{\kappa(t)})^2 \right\|_{\dot{H}^{1-2\varrho}} ds \\
		\leq \ &  C \tau^{\theta} \|\bar{v}_{\kappa(t)} \|_{\dot{H}^{\varrho}} 
		+ C(\varrho) G(\kappa(t)) \tau^{\varrho} \| v_{\kappa(t)} \|_{L^4}^2 \\
		\leq \ &  C \tau^{\theta} \| v_{\kappa(t)} \|_{\dot{H}^{\varrho}}
		+ C \tau^{\theta} \| \mathcal{O}^N_{\kappa(t)} \|_{\dot{H}^{\varrho}}
		+ C(\varrho) G(\kappa(t)) \tau^{\varrho} \| v_{\kappa(t)} \|_{L^4}^2. 
	\end{aligned}
\end{equation}
It can be seen from \eqref{Eq:vbar L4} and \eqref{Eq:vbar-vbark L4} that  
\begin{equation}\label{Est: bound of v}
	\begin{aligned}
		&G(t) \|v_t+v_{\kappa(t)}\|_{L^4} \|v_t-v_{\kappa(t)}\|_{L^4} \\
		= \ & G(t) \|\bar{v}_t+\mathcal{O}_t^N+v_{\kappa(t)}\|_{L^4} \|\bar{v}_t-\bar{v}_{\kappa(t)}+\mathcal{O}_t^N-\mathcal{O}_{\kappa(t)}^N\|_{L^4} \\
		\leq \ & C(\varrho) G(t) \left( \| v_{\kappa(t)}\|_{\dot{H}^{\varrho}} + \| \mathcal{O}^N_{\kappa(t)}\|_{\dot{H}^{\varrho}} + \tau^{\varrho}   G(\kappa(t)) \| v_{\kappa(t)} \|_{L^4}^2 + \|\mathcal{O}_t^N\|_{L^4} \right)   \\ 
		& \times \left( \tau^{\theta} \| v_{\kappa(t)} \|_{\dot{H}^{\varrho}}
		+ \tau^{\theta} \| \mathcal{O}^N_{\kappa(t)} \|_{\dot{H}^{\varrho}}
		+ \tau^{\varrho} G(\kappa(t))\| v_{\kappa(t)} \|_{L^4}^2 + \|\mathcal{O}_t^N-\mathcal{O}_{\kappa(t)}^N\|_{L^4}\right). 
	\end{aligned}
\end{equation}
Recalling the definition of $G(t)$ and the Sobolev embedding inequality $\dot{H}^{\varrho} \hookrightarrow L^4(\mathcal{D})$, we have 
\begin{equation*}
	\tau^{\varrho} G(\kappa(t)) \| v_{\kappa(t)} \|_{L^4}^2 
	= \frac{\tau^{\varrho-\theta} \tau^{\theta} \| v_{\kappa(t)} \|_{L^4}^2}{1+\tau^{\theta}\|v_{\kappa(t)}\|_{\dot{H}^{\varrho}} ^2 + \tau^{\theta}\|\mathcal{O}^N_{\kappa(t)}\|_{\dot{H}^{\varrho}}^2 } \leq C \tau^{\varrho-\theta}. 
\end{equation*}
Substituting the above estimate into \eqref{Est: bound of v}, then by $ab \leq \frac{1}{2} a^2 + \frac{1}{2} b^2$ and noting that $\varrho - \theta > \varrho - 2\theta \geq \frac{1}{4}$ and $\tau<1$, we obtain  
\begin{equation}\label{Eq:G v vbar}
	\begin{aligned}
		& G(t) \|v_t+v_{\kappa(t)}\|_{L^4} \|v_t-v_{\kappa(t)}\|_{L^4} \\
		\leq \ &  C(\varrho) G(t) \left( \| v_{\kappa(t)}\|_{\dot{H}^{\varrho}} + \| \mathcal{O}^N_{\kappa(t)}\|_{\dot{H}^{\varrho}} + C \tau^{\varrho-\theta} + \|\mathcal{O}_t^N\|_{L^4} \right) \\
		&\times \left( \tau^{\theta} \| v_{\kappa(t)} \|_{\dot{H}^{\varrho}}
		+ \tau^{\theta} \| \mathcal{O}^N_{\kappa(t)} \|_{\dot{H}^{\varrho}}
		+ C \tau^{\varrho-\theta} 
		+ \|\mathcal{O}_t^N-\mathcal{O}_{\kappa(t)}^N\|_{L^4} \right) \\ 
		\leq \ &  \frac{C(\varrho) \tau^{\theta} G(t)}{2} \left( \| v_{\kappa(t)}\|_{\dot{H}^{\varrho}} + \| \mathcal{O}^N_{\kappa(t)}\|_{\dot{H}^{\varrho}} + \|\mathcal{O}_t^N\|_{L^4} + 1 \right)^2 \\
		&+ \frac{C(\varrho) \tau^{\theta} G(t)}{2} \left( \| v_{\kappa(t)} \|_{\dot{H}^{\varrho}} 
		+ \| \mathcal{O}^N_{\kappa(t)} \|_{\dot{H}^{\varrho}} 
		+ \tau^{-\theta}\|\mathcal{O}_t^N-\mathcal{O}_{\kappa(t)}^N\|_{L^4} + 1 \right)^2  \\ 
		\leq \ &  C(\varrho) \tau^{\theta} G(t)  \left( \| v_{\kappa(t)} \|_{\dot{H}^{\varrho}}^2 
		+ \| \mathcal{O}^N_{\kappa(t)} \|_{\dot{H}^{\varrho}}^2 
		+ 1 \right) 
		+ C(\varrho) \|\mathcal{O}_t^N\|^2_{L^4} + C(\varrho) \tau^{-\theta}\|\mathcal{O}_t^N-\mathcal{O}_{\kappa(t)}^N\|_{L^4}^2 \\
		\leq \ & C(\varrho) + C(\varrho) \|\mathcal{O}_t^N\|^2_{L^4} + C(\varrho) \tau^{-\theta}\|\mathcal{O}_t^N-\mathcal{O}_{\kappa(t)}^N\|_{L^4}^2. 
	\end{aligned}
\end{equation}
Consequently, we deduce from \eqref{Eq:derivetive of vbar} and \eqref{Eq:G v vbar} that 
\begin{align*}
	\frac{d}{dt} \|\tilde{v}_t\|^2 
	&\leq -\frac{1}{2} \|\tilde{v}_t\|_{\dot{H}^1}^2 
	+ C \left(1+\|\mathcal{O}_t^N\|^4_{L^4} +  \tau^{-2\theta}\|\mathcal{O}_t^N-\mathcal{O}_{\kappa(t)}^N\|_{L^4}^4\right) \\
	& \quad + 4\|\mathbb{O}^{N,\eta}_t\|_{L^\infty}^2\|\tilde{v}_t\|^2+ \|\mathbb{O}^{N,\eta}_t\|_{L^4}^4 + 2\eta^2 \|\mathbb{O}^{N,\eta}_t\|^2. 
\end{align*}
By Gronwall's inequality, 
\begin{align*}
	\|\tilde{v}_t\|^2 \leq C(T,\eta) e^{4\int_{0}^{T}\|\mathbb{O}^{N,\eta}_t\|_{L^\infty}^2 dt} \Big( 1+\|u_0\|^2 + \int_{0}^{T} \|\mathbb{O}^{N,\eta}_t\|_{L^4}^4 dt \\
	+ \int_{0}^{T} \|\mathcal{O}^{N}_t\|_{L^4}^4 dt
	+ \tau^{-2\theta} \int_{0}^{T} \|\mathcal{O}_t^N-\mathcal{O}_{\kappa(t)}^N\|_{L^4}^4 dt \Big). 
\end{align*}
Taking $p/2$ moment on the above inequality, then by H\"older's inequality, Lemmas \ref{Le:exponential int O}, \ref{Le:O eta}, and \ref{Le:O L4}, we get 
\begin{equation}\label{Est:estimate of bar v} 
	\mathbb{E} \big[\sup_{t \in [0,T]} \|\tilde{v}_t\|^2 \big] + \int_{0}^{T} \mathbb{E}[\|\tilde{v}_t\|_{\dot{H}^1}^2] dt
	\leq C(T,\eta,H,u_0),  
\end{equation}
This gives the estimate of numerical approximations $\tilde{v}_t$, $t \in [0, T]$. 

\subsection{Convergence of the full discretization}
In order to investigate the convergence of $v_t$ to $u^N(t)$, we denote $\mathcal{E}^N(t) := u^N(t) - v_t$, which solves
\begin{equation*} 
	\frac{d}{dt} \mathcal{E}^N(t) = -A_N \mathcal{E}^N(t) + P_N f(u^N(t)) - G(t) P_N f(v_{\kappa(t)}), \quad  \mathcal{E}^N(0)=0. 
\end{equation*}
By using the identity $\frac{1}{2} \frac{d}{dt} \|\mathcal{E}^N(t)\|^2 = \langle \mathcal{E}^N(t), \frac{d}{dt}\mathcal{E}^N(t) \rangle$, $0 \leq 1-G(t) < 1$ and $ab \leq \frac{1}{2} a^2 + \frac{1}{2} b^2$, we have 
\begin{align*} 
	\frac{1}{2} \frac{d}{dt} \|\mathcal{E}^N(t)\|^2 
	&= -\|\mathcal{E}^N(t)\|_{\dot{H}^1}^2 - \frac{1}{2} \left\langle \frac{\partial}{\partial x} \mathcal{E}^N(t) , (u^N(t))^2 - (v_t)^2 \right\rangle \\
	&\quad - \frac{1}{2} \left\langle \frac{\partial}{\partial x} \mathcal{E}^N(t) , (v_t)^2 - (v_{\kappa(t)})^2 \right\rangle
	- \frac{1}{2} \left\langle \frac{\partial}{\partial x} \mathcal{E}^N(t) , (v_{\kappa(t)})^2 - G(t) (v_{\kappa(t)})^2 \right\rangle \\
	&\leq -\|\mathcal{E}^N(t)\|_{\dot{H}^1}^2 + \frac{1}{4} \|\mathcal{E}^N(t)\|_{\dot{H}^1}^2 + \frac{1}{4} \|(u^N(t) + v_t) \mathcal{E}^N(t)\|^2
	+ \frac{1}{4} \|\mathcal{E}^N(t)\|_{\dot{H}^1}^2 \\
	&\quad + \frac{1}{4} \left\| (v_t + v_{\kappa(t)} ) (v_t - v_{\kappa(t)}) \right\|^2 
	+ \frac{1}{4} \left(1-G(t)\right) \|\mathcal{E}^N(t)\|_{\dot{H}^1}^2 + \frac{1}{4} \left(1-G(t)\right) \| (v_{\kappa(t)})^2 \|^2 \\
	&\leq \frac{1}{2}\left( \|u^N(t)\|_{L^\infty}^2 + \|v_t\|_{L^\infty}^2\right) \|\mathcal{E}^N(t)\|^2 
	+ \frac{1}{2} \left( \|v_t\|_{L^4}^2 + \|v_{\kappa(t)}\|_{L^4}^2 \right) \left\| v_t - v_{\kappa(t)}\right\|_{L^4}^2 \\
	&\quad + \frac{1}{4} \left( \tau^{\frac{\theta}{2}}\|v_{\kappa(t)}\|_{\dot{H}^{\varrho}} + \tau^{\frac{\theta}{2}}\|\mathcal{O}^N_{\kappa(t)}\|_{\dot{H}^{\varrho}} \right) \| v_{\kappa(t)} \|_{L^4}^4, 
\end{align*}
in which $\tau^{\frac{\theta}{2}}a \leq \tau^{\theta} a^2 + \frac{1}{4}$ and the following estimate is used,  
\begin{equation*}
	1-G(t) \leq  \frac{\tau^{\theta}\|v_{\kappa(t)}\|_{\dot{H}^{\varrho}} ^2 + \tau^{\theta}\|\mathcal{O}^N_{\kappa(t)}\|_{\dot{H}^{\varrho}}^2}{1+\tau^{\frac{\theta}{2}}\|v_{\kappa(t)}\|_{\dot{H}^{\varrho}} - \frac{1}{4} + \tau^{\frac{\theta}{2}}\|\mathcal{O}^N_{\kappa(t)}\|_{\dot{H}^{\varrho}} - \frac{1}{4} } 
	\leq \tau^{\frac{\theta}{2}}\|v_{\kappa(t)}\|_{\dot{H}^{\varrho}} + \tau^{\frac{\theta}{2}}\|\mathcal{O}^N_{\kappa(t)}\|_{\dot{H}^{\varrho}}. 
\end{equation*}
For convenience, we denote
\begin{align*}
	\mathcal{G}_t 
	&:= \int_{0}^{t} \left( \|v_s\|_{L^4}^2 + \|v_{\kappa(s)}\|_{L^4}^2 \right) \left\| v_s - v_{\kappa(s)}\right\|_{L^4}^2 ds
	+ \frac{1}{2} \int_{0}^{t} \left( \tau^{\frac{\theta}{2}}\|v_{\kappa(s)}\|_{\dot{H}^{\varrho}} + \tau^{\frac{\theta}{2}}\|\mathcal{O}^N_{\kappa(s)}\|_{\dot{H}^{\varrho}} \right) \| v_{\kappa(s)} \|_{L^4}^4 ds. 
\end{align*}
Then it follows from Gronwall's inequality and $\mathcal{E}^N(0) = 0$ that 
\begin{equation}\label{Est:error eN}
	\|\mathcal{E}^N(t)\|^2 \leq \mathcal{G}_t \exp \left\{ \int_{0}^{t} \|u^N(s)\|_{L^\infty}^2 ds + \int_{0}^{t} \|v_s\|_{L^\infty}^2 ds \right\}. 
\end{equation}
By H\"older's inequality and $ab \leq \frac{1}{2} a^2 + \frac{1}{2} b^2$, we have 
\begin{equation}\label{Eq:mathcal G 1}
	\begin{aligned}
		\mathcal{G}_t
		&\leq \sqrt{2} \left( \int_{0}^{t} \left( \|v_s\|_{L^4}^4 + \|v_{\kappa(s)}\|_{L^4}^4 \right) ds \right)^{\frac{1}{2}} 
		\left( \int_{0}^{t}  \left\| v_s - v_{\kappa(s)}\right\|_{L^4}^4 ds  \right)^{\frac{1}{2}} \\
		&\quad + \frac{\tau^{\frac{\theta}{2}}}{2} \int_{0}^{t} \left( \|v_{\kappa(s)}\|_{\dot{H}^{\varrho}}^2 + \|\mathcal{O}^N_{\kappa(s)}\|_{\dot{H}^{\varrho}}^2 \right) ds + 
		\frac{\tau^{\frac{\theta}{2}}}{4} \int_{0}^{t} \| v_{\kappa(s)} \|_{L^4}^8 ds. 
	\end{aligned}
\end{equation}
Before making further estimations of $\mathcal{E}^N(t)$, we first give three lemmas involving the estimates of $v_t - v_{\kappa(t)}$, $v_t$, $v_{\kappa(t)}$, and the stochastic convolutions. 
We denote 
\begin{equation}\label{Theta}
	\Theta_1(t) := \int_{\kappa(t)}^{t} S_N(t-s) P_N dB^H(s), \quad \Theta_2(t) := \int_{0}^{\kappa(t)} \left( S_N(t-s) - S_N(\kappa(t)-s) \right) P_N dB^H(s). 
\end{equation}

\begin{Lemma}\label{Le:v1}
	For $v_t$ given by \eqref{Eq:continuous version} and $\theta$ given by \eqref{rho and theta}, it holds that 
	\begin{align*}
		\|v_t - v_{\kappa(t)}\|_{L^4} \leq C \tau^{2\theta} \|v_{t_{0}}\|_{\dot{H}^{\frac{1}{4}+4\theta}} 
		+ C(\theta, T) \tau^{2\theta}
		+ C \tau^{\frac{3}{8} - \theta} + \left\| \Theta_1(t) \right\|_{L^4} 
		+ \left\| \Theta_2(t) \right\|_{L^4},  
	\end{align*}
	where the constants $C$ and $C(\theta, T)$ are independent of $N$ and $\tau$. 
\end{Lemma}
\begin{proof}
	From \eqref{Eq:continuous version} and by the Sobolev embedding inequality $\dot{H}^{\frac{1}{4}} \hookrightarrow L^4(\mathcal{D})$, we have  
	\begin{align*}
		\|v_t - v_{\kappa(t)}\|_{L^4} 
		&\leq C \left\| S_N(\kappa(t)) \left( S_N(t-\kappa(t)) - I \right) v_{t_{0}} \right\|_{\dot{H}^{\frac{1}{4}}} \\
		&\quad + C \int_{0}^{\kappa(t)} G(s) \left\| S_N(\kappa(t)-s) \left( S_N(t-\kappa(t)) - I \right) f(v_{\kappa(s)}) \right\|_{\dot{H}^{\frac{1}{4}}} ds \\
		&\quad + C \int_{\kappa(t)}^{t} G(s) \left\| S_N(t-s) f(v_{\kappa(s)}) \right\|_{\dot{H}^{\frac{1}{4}}} ds + \| \Theta_1(t) \|_{L^4} + \| \Theta_2(t) \|_{L^4} \\
		& =: I_1 + I_2 + I_3 + \| \Theta_1(t) \|_{L^4} + \| \Theta_2(t) \|_{L^4}. 
	\end{align*}
	For $I_1$, through \eqref{Est:semigroup 1} with $\rho = 2 \theta$, we have 
	\begin{align*}
		I_1
		= C \left\| A^{-2 \theta} \left( S_N(t-\kappa(t)) - I \right) S_N(\kappa(t)) A^{\frac{1}{8}+2\theta} v_{t_{0}} \right\| 
		\leq C \tau^{2\theta} \|v_{t_{0}}\|_{\dot{H}^{\frac{1}{4}+4\theta}} . 
	\end{align*}
	For $I_2$, by \eqref{Est:S B 2} with $\zeta = 6 \theta + \frac{1}{4}$, $\dot{H}^{\varrho} \hookrightarrow L^4(\mathcal{D})$ and noting that  $\theta < \frac{1}{8}$, we have 
	\begin{align*}
		I_2&= \frac{C}{2} \int_{0}^{\kappa(t)} G(s) \left\| A^{-3\theta} \left( S_N(t-\kappa(t)) - I \right) A^{3\theta+\frac{1}{8}} S_N(\kappa(t)-s) \frac{\partial}{\partial x} (v_{\kappa(s)})^2 \right\| ds \\
		&\leq C(\theta) \tau^{3\theta} \int_{0}^{\kappa(t)} G(s) \left( \kappa(t)-s \right)^{-(3\theta + \frac{1}{8} + \frac{1}{2})} \left\| \frac{\partial}{\partial x} (v_{\kappa(s)})^2 \right\|_{W^{-1, 2}} ds \\
		&= C(\theta) \tau^{2\theta} \int_{0}^{\kappa(t)} \left( \kappa(t)-s \right)^{-(3\theta + \frac{5}{8})} \frac{\tau^{\theta} \| v_{\kappa(s)} \|_{L^4}^2}{1+\tau^{\theta}\|v_{\kappa(s)}\|_{\dot{H}^{\varrho}} ^2 + \tau^{\theta}\|\mathcal{O}^N_{\kappa(s)}\|_{\dot{H}^{\varrho}} ^2 } ds \\
		&\leq C(\theta) \tau^{2\theta} \int_{0}^{\kappa(t)} \left( \kappa(t)-s \right)^{-(3\theta + \frac{5}{8})}  ds 
		\leq C(\theta, T) \tau^{2\theta}. 
	\end{align*}
	For $I_3$, taking $\zeta=\frac{1}{4}$ in \eqref{Est:S B 2}, we have 
	\begin{align*}
		I_3 &\leq C \int_{\kappa(t)}^{t} (t-s)^{-\frac{5}{8}} \frac{ \|v_{\kappa(s)}\|_{L^4}^2}{1+\tau^{\theta}\|v_{\kappa(s)}\|_{\dot{H}^{\varrho}} ^2 + \tau^{\theta}\|\mathcal{O}^N_{\kappa(s)}\|_{\dot{H}^{\varrho}} ^2 } ds \\
		&= C \tau^{-\theta} \int_{\kappa(t)}^{t} (t-s)^{-\frac{5}{8}} \frac{\tau^{\theta} \|v_{\kappa(s)}\|_{L^4}^2}{1+\tau^{\theta}\|v_{\kappa(s)}\|_{\dot{H}^{\varrho}}^2 + \tau^{\theta}\|\mathcal{O}^N_{\kappa(s)}\|_{\dot{H}^{\varrho}} ^2 } ds 
		\leq C \tau^{\frac{3}{8} - \theta}, 
	\end{align*}
	which completes the proof. 
\end{proof}

\begin{Lemma}\label{Le:O L4}
	For $\Theta_1(t)$ and $\Theta_2(t)$ given by \eqref{Theta} and for any $t \in [0,T]$, it holds that 
	\begin{align*}
		\mathbb{E} \left[\left\| \Theta_1(t) \right\|_{L^4}^4\right] \leq C(H) \tau^{4H - 1 - 2\epsilon} \quad \text{and} \quad \mathbb{E} \left[\left\| \Theta_2(t) \right\|_{L^4}^4\right] \leq C(H) \tau^{4H - 1 - 2\epsilon},
	\end{align*}
	where $\epsilon$ is a small positive number. 
\end{Lemma}
\begin{proof}
	Since the stochastic convolution is a $U$-valued Gaussian random variable, we have  
	\begin{equation*}
		\mathbb{E} \left[\left\| \Theta_1(t) \right\|_{L^4}^4 \right]
		= \mathbb{E} \left[|Y|^4\right] \int_{0}^{1} \left( \mathbb{E} \left[(\Theta_1(t))^2\right] \right)^2 dx, 
	\end{equation*}
	in which $Y$ stands for the standard Gaussian random variable. By the independence of $\{w_k^H\}_{k=1}^{N}$, It\^o's isometry and Lemma \ref{Le:S(t)} with $\vartheta = \frac{1}{4} + \frac{\epsilon}{2}$,  we obtain 
	\begin{align*}
		\mathbb{E} \left[(\Theta_1(t))^2\right] 
		&= \mathbb{E} \left[\left( \sum_{k=1}^{N} \phi_k(x) \int_{\kappa(t)}^{t} e^{-\lambda_k(t-r)} dw^H_k (r) \right)^2\right] 
		= \sum_{k=1}^{N} \phi_k^2(x) \mathbb{E} \left[\left( \int_{\kappa(t)}^{t} e^{-\lambda_k(t-r)} dw^H_k (r) \right)^2\right] \\
		&\leq \sum_{k=1}^{N}  \int_{\kappa(t)}^{t} \int_{\kappa(t)}^{t} \left\langle S_N(t-u) \phi_k , S_N(t-v) \phi_k \right\rangle \phi(u-v) du dv  \\
		&\leq C(H) \sum_{k=1}^{N} \lambda_k^{-\frac{1}{2}-\epsilon}  \left( t-\kappa(t) \right)^{2(H-\frac{1}{4}-\frac{\epsilon}{2})}
		\leq C(H) \tau^{2H - \frac{1}{2} - \epsilon},  
	\end{align*}
	where $\epsilon$ is an arbitrarily small positive number. 
	A similar derivation leads to 
	\begin{equation*}
		\mathbb{E} \left[\left\|\Theta_2(t)\right\|_{L^4}^4\right] = \mathbb{E} \left[|Y|^4\right] \int_{0}^{1} \left( \mathbb{E} \left[(\Theta_2(t))^2\right] \right)^2 dx. 
	\end{equation*}
	Through the independence of $\{w_k^H\}_{k=1}^{N}$, It\^o's isometry, Lemma \ref{Le:S(t)}, and the fact that $x^{-a}(1-e^{-x}) \leq 1$, $\forall \, a \in (0,1)$, we have 
	\begin{align*}
		&\mathbb{E} \left[(\Theta_2(t))^2\right] \\
		= \ &  \sum_{k=1}^{N} \int_{0}^{\kappa(t)} \int_{0}^{\kappa(t)} e^{-\lambda_k(\kappa(t)-u)} \left( e^{-\lambda_k(t-\kappa(t))} - 1 \right) e^{-\lambda_k(\kappa(t)-v)} \left( e^{-\lambda_k(t-\kappa(t))} - 1 \right) \phi(u-v) du dv \\
		= \ &  \sum_{k=1}^{N} \lambda_k^{-2H} \left( e^{-\lambda_k(t-\kappa(t))} - 1 \right)^2 \int_{0}^{\kappa(t)} \int_{0}^{\kappa(t)} \left\langle A^{H} S_N(\kappa(t)-u)  \phi_k, A^{H} S_N(\kappa(t)-v) \phi_k \right\rangle  \phi(u-v) du dv \\
		\leq \ &  C(H) \sum_{k=1}^{N} \left( e^{-\lambda_k(t-\kappa(t))} - 1 \right)^2 \lambda_k^{-2H} \|\phi_k\|^2  \\
		= \ &  C(H) \sum_{k=1}^{N} \lambda_k^{-\frac{1}{2}-\epsilon} \left[\lambda_k^{-(H-\frac{1}{4}-\frac{\epsilon}{2})} \left( 1-e^{-\lambda_k(t-\kappa(t))} \right)\right]^2 \leq C(H) \tau^{2H-\frac{1}{2}-\epsilon}, 
	\end{align*}
	which completes the proof.  
\end{proof}

\begin{Lemma}\label{Le:v2}
	For $v_t$ given by \eqref{Eq:continuous version} and $\varrho$ given by \eqref{rho and theta}, it holds that 
	\begin{equation}\label{Est:H varrho 2 of v}
		\begin{aligned}
			\|v_{\kappa(t)}\|_{\dot{H}^{\varrho}} 
			&\leq C \|v_{t_{0}}\|_{\dot{H}^{\varrho}} + C(\varrho) \int_{0}^{\kappa(t)} (\kappa(t)-s)^{-\frac{2\varrho+3}{4}} \|\bar{v}_{\kappa(s)}\|^2 ds \\
			&\quad + C(\varrho, T)
			\left( \int_{0}^{t}  \|\mathcal{O}^N_{\kappa(s)}\|^{\frac{4\varrho+14}{1-2\varrho}} ds \right)^{\frac{1-2\varrho}{2\varrho+7} }
			+ \| \mathcal{O}^N_{\kappa(t)} \|_{\dot{H}^{\varrho}},  
		\end{aligned}
	\end{equation}
	and 
	\begin{equation}\label{Est:L4 of v}
		\|v_t\|_{L^4} 
		\leq C \|v_{t_0}\|_{\dot{H}^{\frac{1}{4}}} +  C \int_{0}^{t} (t-s)^{-\frac{7}{8}} \|v_{\kappa(s)} \|^2  ds + \| \mathcal{O}^N_{t} \|_{L^4},  
	\end{equation}
	in which the constants $C$, $C(\varrho)$ and $C(\varrho, T)$ are independent of $N$ and $\tau$. 
\end{Lemma}
\begin{proof}
	From \eqref{Eq:continuous version} and by taking $\zeta=\varrho$ in \eqref{Est:S B 1}, we know that
	\begin{align*}
		\|v_t\|_{\dot{H}^{\varrho}} 
		&\leq C \|v_{t_{0}}\|_{\dot{H}^{\varrho}} + C(\varrho) \int_{0}^{t} G(s) (t-s)^{-\frac{\varrho}{2}-\frac{3}{4}} \left\| \frac{\partial}{\partial x} (v_{\kappa(s)})^2 \right\|_{W^{-1, 1}} ds + \| \mathcal{O}^N_{t} \|_{\dot{H}^{\varrho}} \\
		&= C \|v_{t_{0}}\|_{\dot{H}^{\varrho}} + C(\varrho) \int_{0}^{t} G(s) (t-s)^{-\frac{\varrho}{2}-\frac{3}{4}} \|v_{\kappa(s)} \|^2 ds + \| \mathcal{O}^N_{t} \|_{\dot{H}^{\varrho}}. 
	\end{align*}
	Substituting $t = \kappa(t)$ into the above estimate, then recalling $G(s)<1$, $\bar{v}_t = v_t - \mathcal{O}^N_{t}$, $\frac{1}{4} \leq \varrho \leq \frac{3}{8}$ and using H\"older's inequality, we have  
	\begin{align*}
		\|v_{\kappa(t)}\|_{\dot{H}^{\varrho}} 
		&\leq C \|v_{t_{0}}\|_{\dot{H}^{\varrho}} + C(\varrho) \int_{0}^{\kappa(t)}  (\kappa(t)-s)^{-\frac{2\varrho+3}{4}} \|v_{\kappa(s)}\|^2   ds + \| \mathcal{O}^N_{\kappa(t)} \|_{\dot{H}^{\varrho}} \\
		&\leq C \|v_{t_{0}}\|_{\dot{H}^{\varrho}} + 2C(\varrho) \int_{0}^{\kappa(t)} (\kappa(t)-s)^{-\frac{2\varrho+3}{4}} \|\bar{v}_{\kappa(s)}\|^2 ds \\ 
		&\quad + 2C(\varrho) \left(\frac{8}{1-2\varrho}\right)^{\frac{4\varrho+6}{2\varrho+7}} T^{\frac{(1-2\varrho)(2\varrho+3)}{4(2\varrho+7)}}
		\left( \int_{0}^{t}  \|\mathcal{O}^N_{\kappa(s)}\|^{\frac{4\varrho+14}{1-2\varrho}} ds \right)^{\frac{1-2\varrho}{2\varrho+7} } 
		+ \| \mathcal{O}^N_{\kappa(t)} \|_{\dot{H}^{\varrho}}, 
	\end{align*}
	which verifies \eqref{Est:H varrho 2 of v}. 
	Finally, by $\dot{H}^{\frac{1}{4}} \hookrightarrow L^4(\mathcal{D})$ and taking $\zeta=\frac{1}{4}$ in \eqref{Est:S B 1}, it is easy to get \eqref{Est:L4 of v}. 	
\end{proof}
Recalling $\tilde{v}_t = v_t - \mathcal{O}^N_{t}$ and using H\"older's inequality, then through \eqref{Est:L4 of v}, we obtain 
\begin{equation}\label{Est:L4 4 of v}
	\begin{aligned}
		\|v_t\|_{L^4}^4 
		&\leq C \|v_{t_0}\|_{\dot{H}^{\frac{1}{4}}}^4 + C \left( \int_{0}^{t} (t-s)^{-\frac{7}{8}} \|\tilde{v}_{\kappa(s)}\|^2  ds \right)^4 \\
		&\quad + C \left( \int_{0}^{t} (t-s)^{-\frac{15}{16}} ds \right)^{\frac{56}{15}}
		\left( \int_{0}^{t}  \|\mathbb{O}^{N,\eta}_{\kappa(s)} \|^{30} ds \right)^{\frac{4}{15}} 
		+ C \| \mathcal{O}^N_{t} \|_{L^4}^4, 
	\end{aligned}
\end{equation}
and 
\begin{equation}\label{Est:L4 8 of v}
	\begin{aligned}
		\|v_{\kappa(t)}\|_{L^4}^8 
		&\leq C \|v_{t_0}\|_{\dot{H}^{\frac{1}{4}}}^8 + C \left( \int_{0}^{\kappa(t)} (\kappa(t)-s)^{-\frac{7}{8}} \|\tilde{v}_{\kappa(s)}\|^2  ds \right)^8 \\
		&\quad + C \left( \int_{0}^{\kappa(t)} (\kappa(t)-s)^{-\frac{15}{16}} ds \right)^{\frac{112}{15}}
		\left( \int_{0}^{\kappa(t)}  \|\mathbb{O}^{N,\eta}_{\kappa(s)} \|^{30} ds \right)^{\frac{8}{15}} 
		+ C \| \mathcal{O}^N_{\kappa(t)} \|_{L^4}^8. 
	\end{aligned}
\end{equation}

For $\mathcal{R}>0$ and $\varrho$ given by \eqref{rho and theta}, and as usual setting  $\inf \varnothing = + \infty$, we introduce a sequence of stopping times $\sigma_i$ as follows: 
\begin{align*}
	\sigma_1 &:= T \wedge \inf \{ t \in [0,T] : \int_{0}^{t} \|u^N(s)\|_{L^\infty}^2 ds \geq \mathcal{R} \} \\
	\sigma_2 &:= T \wedge \inf \{ t \in [0,T] : \int_{0}^{t} \|v_s\|_{L^\infty}^2 ds \geq \mathcal{R} \}, \\
	\sigma_3 &:= T \wedge \inf \{ t \in [0,T] : \|\tilde{v}_t\|^2 \geq \mathcal{R} \}, \\
	\sigma_4 &:= T \wedge \inf \{ t \in [0,T] : \int_{0}^{t} \| \mathcal{O}_{\kappa(s)}^N \|_{\dot{H}^{\varrho}}^2 ds \geq \mathcal{R} \}, \\
	\sigma_5 &:= T \wedge \inf \{ t \in [0,T] : \int_{0}^{t} \| \mathbb{O}_{\kappa(s)}^{N,\eta} \|^{\frac{4\varrho+14}{1-2\varrho}} ds \geq \mathcal{R} \}, \\
	\sigma_6 &:= T \wedge \inf \{ t \in [0,T] : \int_{0}^{t} \| \mathcal{O}_{\kappa(s)}^N \|_{L^4}^{8} ds \geq \mathcal{R} \}, \\
	\sigma_7 &:= T \wedge \inf \{ t \in [0,T] : \int_{0}^{t} \| \mathcal{O}_{s}^N \|_{L^4}^{4} ds \geq \mathcal{R} \}. 
\end{align*}
Denote $\sigma := \min\limits_{i = 1, 2, ..., 7} \left\{ \sigma_i \right\}$. Obviously $\sigma$ depends on $N$, $\mathcal{R}$ and $\tau$. 
Through \eqref{Est:L4 4 of v} we have 
\begin{align*}
	\sup\limits_{t \leq \sigma} \int_{0}^{t} \|v_s\|_{L^4}^4 ds
	&\leq C(T) \|v_{t_0}\|_{\dot{H}^{\frac{1}{4}}}^4 + C(\mathcal{R}) \sup\limits_{t \leq \sigma} \int_{0}^{t} \left(  \int_{0}^{s} (s-r)^{-\frac{7}{8}} dr \right)^4 ds \\
	&\quad + C(\mathcal{R}) \sup\limits_{t \leq \sigma} \int_{0}^{t}  \left( \int_{0}^{s} (s-r)^{-\frac{15}{16}} dr \right)^{\frac{56}{15}} ds
	+ C \sup\limits_{t \leq \sigma} \int_{0}^{t} \| \mathcal{O}^N_{s} \|_{L^4}^4 ds \\
	&\leq C(T) \|v_{t_0}\|_{\dot{H}^{\frac{1}{4}}}^4 + C(\mathcal{R}, T). 
\end{align*}
Similarly, through \eqref{Est:L4 8 of v} and \eqref{Est:H varrho 2 of v}, we obtain 
\begin{align*}
	\sup\limits_{t \leq \sigma} \int_{0}^{t} \|v_{\kappa(s)}\|_{L^4}^8 ds
	&\leq C(T) \|v_{t_0}\|_{\dot{H}^{\frac{1}{4}}}^8 + C(\mathcal{R}, T),  \\
	\sup\limits_{t \leq \sigma} \int_{0}^{t} \|v_{\kappa(s)}\|_{\dot{H}^{\varrho}}^2 ds
	&\leq C(T) \|v_{t_0}\|_{\dot{H}^{\varrho}}^2 + C(\mathcal{R}, T, \varrho).  
\end{align*}
Substituting the above estimates into \eqref{Eq:mathcal G 1}, we obtain 
\begin{equation}\label{Eq:mathcal G 2}
	\begin{aligned}
		\sup\limits_{t \leq \sigma} \mathcal{G}_t
		&\leq \left( C(T)   \|v_{t_0}\|_{\dot{H}^{\frac{1}{4}}}^2 + C(\mathcal{R}, T) \right)
		\left( \int_{0}^{\sigma}  \left\| v_s - v_{\kappa(s)}\right\|_{L^4}^4 ds  \right)^{\frac{1}{2}} \\
		&\quad + \tau^{\frac{\theta}{2}} \left( C(T) \|v_{t_0}\|_{\dot{H}^{\varrho}}^2 + C(\mathcal{R}, T, \varrho) + C(T) \|v_{t_0}\|_{\dot{H}^{\frac{1}{4}}}^8 + C(\mathcal{R}, T) \right).  
	\end{aligned}
\end{equation}
Through Lemma \ref{Le:v1}, we know that 
\begin{equation}\label{Eq:int vs-vks}
	\begin{aligned}
		\int_{0}^{\sigma} \|v_s - v_{\kappa(s)}\|_{L^4}^4 ds &\leq C(\theta, T)\left( \tau^{8\theta} \|v_{t_{0}}\|^4_{\dot{H}^{\frac{1}{4}+4\theta}} 
		+ \tau^{8\theta}
		+ \tau^{\frac{3}{2} - 4\theta} \right)
		+ C \int_{0}^{\sigma} \left(\left\| \Theta_1(s) \right\|_{L^4}^4 + \left\| \Theta_2(s) \right\|_{L^4}^4\right) ds.  
	\end{aligned}
\end{equation}
Accordingly, it follows from \eqref{Est:error eN}, \eqref{Eq:mathcal G 2}, and \eqref{Eq:int vs-vks} that 
\begin{equation}\label{Est:estimate of eN}
	\begin{aligned}
		\sup\limits_{t \leq \sigma} \|\mathcal{E}^N(t)\|^2 
		\leq \ &  \tau^{\frac{\theta}{2}} C(\mathcal{R}, T, \varrho) \left( \|v_{t_0}\|_{\dot{H}^{\varrho}}^2
		+ \|v_{t_0}\|_{\dot{H}^{\frac{1}{4}}}^8 +  1 \right) 
		+ C(\mathcal{R}, T, \theta) \left( \|v_{t_0}\|_{\dot{H}^{\frac{1}{4}}}^2 + 1 \right) \times \\
		& \left( \tau^{\min\{ 8\theta , \frac{3}{2} - 4\theta \}} \left(\|v_{t_0}\|_{\dot{H}^{\frac{1}{4}+4\theta}}^4+1\right) 
		+ \int_{0}^{\sigma} \left(\left\| \Theta_1(s) \right\|_{L^4}^4 + \left\| \Theta_2(s) \right\|_{L^4}^4\right) ds \right)^{\frac{1}{2}}. 
	\end{aligned}
\end{equation}

Next we show that $\lim\limits_{\mathcal{R} \rightarrow \infty} \mathbb{P} (\sigma < T) = 0$ holds for any $N$ and $\tau$. It suffices to prove $\lim\limits_{\mathcal{R} \rightarrow \infty} \mathbb{P} (\sigma_i < T) = 0$, $i = 1, 2, ..., 7$,  for any $N$ and $\tau$. 
The stopping time $\sigma_1$ does not involve the step size $\tau$, so it can be proved, similarly to the case of $\tau_1^N$, that $\mathbb{P}(\sigma_1<T) \rightarrow 0$ as $\mathcal{R} \rightarrow \infty$. 
The proofs of $\mathbb{P}(\sigma_i < T) \rightarrow 0$ for $i = 4, 5, 6, 7$ follow directly from the moment boundedness of stochastic convolutions. 
For $\sigma_2$, through \eqref{Est:estimate of bar v} and $\dot{H}^1 \hookrightarrow L^{\infty}(\mathcal{D})$, we know that 
\begin{align*} 
	\int_{0}^{T} \mathbb{E}[\|\tilde{v}_s\|_{L^\infty}^2] ds 
	\leq \int_{0}^{T} \mathbb{E} [\|\tilde{v}_s\|_{\dot{H}^1}^2] ds 
	\leq C(T,\eta,H,u_0). 
\end{align*}
Recalling $\tilde{v}_t = v_t - \mathbb{O}^{N,\eta}_{t}$, then by Markov's inequality and Lemma \ref{Le:O eta}, we obtain 
\begin{align*}
	\mathbb{P}(\sigma_2 < T) 
	&= \mathbb{P}\left( \int_{0}^{T} \|v_s\|_{L^\infty}^2 ds \geq \mathcal{R} \right) 
	\leq \frac{\int_{0}^{T} \mathbb{E} \left[\|v_s\|_{L^\infty}^2\right] ds}{\mathcal{R}}  \\
	&\leq \frac{C(T,\eta,H,u_0)}{\mathcal{R}} + \frac{C \int_{0}^{T} \mathbb{E} \left[\|\mathbb{O}_s^{N,\eta}\|_{L^\infty}^2\right] ds}{\mathcal{R}} \xrightarrow{\mathcal{R} \rightarrow \infty} 0, 
\end{align*}
which implies that for any $N$ and $\tau$, $\mathbb{P}(\sigma_2 < T) \rightarrow 0$ as $\mathcal{R} \rightarrow \infty$. 
For $\sigma_3$, by \eqref{Est:estimate of bar v}, we have 
\begin{align*}
	\mathbb{P}(\sigma_3 < T) &= \mathbb{E}\left[ \chi_{\{\sigma_3 < T\}} \right]
	\leq \mathbb{E}\left[ \chi_{\{\sigma_3 < T\}} \frac{\|\tilde{v}\left(\sigma_3 \wedge T\right)\|^2}{\mathcal{R}} \right] \\
	&\leq \frac{1}{\mathcal{R}} \mathbb{E} \left[ \sup\limits_{t \in [0,T]} \|\tilde{v}_t\|^2 \right]
	\leq \frac{ C(T,\eta,H,u_0) }{\mathcal{R}} 
	\xrightarrow[]{\mathcal{R} \rightarrow \infty} 0. 
\end{align*}

Now we are ready to prove that $\sup\limits_{t \in [0,T]} \|\mathcal{E}^N(t)\|^2 \rightarrow 0$, as $\tau \rightarrow 0$, holds in probability. 
In fact, for any $a>0$, we have 
\begin{equation*}
	\mathbb{P} \left( \sup\limits_{t \in [0,T]}  \|\mathcal{E}^N(t)\|^2 > a \right)  
	\leq \mathbb{P} \left( \sup\limits_{t \in [0,T]}  \|\mathcal{E}^N(t)\|^2 > a \ \text{and} \ \sigma = T \right) 
	+ \mathbb{P} \left( \sigma < T \right). 
\end{equation*}
For any $N \in \mathbb{N}$, $\tau>0$, and $\epsilon>0$, there exists a positive number $\mathcal{R}$ which depends on $\epsilon$ but is independent of $N$ and $\tau$ such that 
\begin{equation*}
	\mathbb{P} (\sigma < T) < \frac{\epsilon}{2}. 
\end{equation*}
We claim that for any $N \in \mathbb{N}$, $\epsilon>0$, and a fixed $\mathcal{R}>0$, there exists $\tau$ depending on $\epsilon$ and $\mathcal{R}$ such that 
\begin{equation}\label{claim1}
	\mathbb{P} \left( \sup\limits_{t \in [0,T]}  \|\mathcal{E}^N(t)\|^2 > a \ \text{and} \ \sigma = T \right) < \frac{\epsilon}{2}. 
\end{equation}
In fact, replacing $\sigma$ by $T$ in \eqref{Est:estimate of eN}, then applying Markov's inequality and Lemma \ref{Le:O L4}, we obtain 
\begin{align*}
	&\mathbb{P} \left( \sup\limits_{t \in [0,T]}  \|\mathcal{E}^N(t)\|^2 > a \ \text{and} \ \sigma = T \right) 
	\leq \frac{1}{a} \mathbb{E} \left[ \sup\limits_{t \in [0,T]}  \|\mathcal{E}^N(t)\|^2 \right] \\
	\leq \ & \frac{\tau^{\frac{\theta}{2}} C(\mathcal{R}, T, \varrho)}{a} \left( \mathbb{E} \left[\|v_{t_0}\|_{\dot{H}^{\varrho}}^2\right]
	+ \mathbb{E} \left[\|v_{t_0}\|_{\dot{H}^{\frac{1}{4}}}^8\right] +  1 \right) 
	+ \frac{C(\mathcal{R}, T, \theta)}{a} \left( \mathbb{E}\left[ \|v_{t_0}\|_{\dot{H}^{\frac{1}{4}}}^4\right] + 1 \right)^{\frac{1}{2}} \times \\
	& \left( \tau^{\min\{8\theta, \frac{3}{2}-4\theta\}} \left(\mathbb{E}\left[\|v_{t_0}\|_{\dot{H}^{\frac{1}{4}+4\theta}}^4\right]+1\right) + \int_{0}^{T} \left(\mathbb{E} \left[\left\| \Theta_1 \right\|_{L^4}^4\right] + \mathbb{E} \left[\left\| \Theta_2 \right\|_{L^4}^4\right] \right) ds \right)^{\frac{1}{2}}\\
	\leq \ &  \frac{\tau^{\frac{\theta}{2}} C(\mathcal{R}, T, \varrho)}{a} + \frac{C(\mathcal{R}, T, \theta, H)}{a} \left( \tau^{\min\{4\theta, \frac{3}{4}-2\theta\}} + \tau^{2H-\frac{1}{2}-\epsilon} \right). 
\end{align*}
This verifies the claim \eqref{claim1}, and thus for $p \geq 1$ and $N \in \mathbb{N}$, 
\begin{equation*}
	\sup\limits_{t \in [0,T]}  \| u^N(t) - v_t \|^p  \rightarrow 0, \quad \text{as } \tau \rightarrow 0, 
\end{equation*}
holds in probability. By virtue of the estimate of $\bar{v}_t$, i.e. \eqref{Est:estimate of bar v}, we know that 
\begin{equation*}
	\mathbb{E} \left[\|v_t\|^p\right] \leq 2^{p-1} \mathbb{E} \left[\|\bar{v}_t\|^p\right] + 2^{p-1} \mathbb{E} \left[\|\mathcal{O}_t^N\|^p\right] 
	\leq C(p, T) \left(\mathbb{E} \left[\|u_{0}\|^p\right] + \mathbb{E} \left[\|\mathcal{O}_t^N\|^p\right] + 1 \right) < \infty.  
\end{equation*}
Through Lemma \ref{Le:probability to expectation}, for $p \geq 1$, we have 
\begin{equation*}
	\sup\limits_{t \in [0,T]} \sup\limits_{N \in \mathbb{N}} \mathbb{E} \left[\left\| v_t - u^N(t) \right\|^p\right] \rightarrow 0, \quad \text{as} \ \tau \rightarrow 0. 
\end{equation*}
Finally, combined with this result and Theorem \ref{thm:convergence of  semi-discrete}, then using the triangle inequality, we obtain the main convergence result of this paper, shown in the following theorem. 
\begin{Theorem}
	Under Assumption \ref{Asp:initial value}, for $p>0$, it holds that 
	\begin{equation*}
		\sup\limits_{t \in [0,T]} \mathbb{E} \left[\| u(t) - v_t \|^p\right] \rightarrow 0, \quad \text{as } N \rightarrow \infty \ \text{and} \ \tau \rightarrow 0,  
	\end{equation*}
	where $u(t)$ and $v_t$ are given by \eqref{Eq:mild solution u1} and \eqref{Eq:continuous version}.  
\end{Theorem}

\section{Numerical experiment}\label{Sec:numerical example}
In this section we conduct a numerical experiment to verify the theoretical results. Since the exact solution of the SBE \eqref{Burgers} is not available, we use numerical approximations with very small step size for reference. The errors are measured in the mean-square sense at the endpoint $T=1$, and the expectations are obtained by the Monte--Carlo method computing average over 1000 trajectories. 
It can be seen from Table \ref{table1} that for different $H$, the strong errors between the full-discrete solution and the mild solution become small for more refined step size $\tau$, which verify the strong convergence of our scheme. Furthermore, convergence rates can be observed, which seem to be related to the Hurst parameter $H$, though it does not proved theoretically. 
\begin{table*}[htbp!]
	\centering
	\caption{ Errors and convergence rates of the fully discrete scheme in time direction by taking $N=1000$ ($N$ is the dimension of the spectral Galerkin projection space). Reference solution is obtained by taking $\tau = \frac{1}{1024}$.}\label{table1}
	\renewcommand\arraystretch{1.4} 
	\resizebox{0.6\textwidth}{!}{\begin{tabular}{c|cc|cc|cc}
			\hline
			\multirow{2}{*}{$\tau$} & \multicolumn{2}{c|}{$H=0.9$} & \multicolumn{2}{c|}{$H=0.7$} & \multicolumn{2}{c}{$H=0.6$} \\  
			& error           & rate     & error           & rate     & error          & rate     \\
			\hline
			\raisebox{0.15ex}{1}/\raisebox{-0.35ex}{4}                 & 2.0513e-02      &          & 4.9763e-02      &          & 6.7938e-02     &          \\[1mm]
			\raisebox{0.15ex}{1}/\raisebox{-0.35ex}{8}                 & 1.4643e-02      & 0.49     & 4.0368e-02      & 0.30     & 5.8234e-02  & 0.22     \\[1mm]
			\raisebox{0.15ex}{1}/\raisebox{-0.35ex}{16}                 & 8.5205e-03      & 0.78     & 2.6524e-02      &  0.61  & 4.0781e-02     &  0.51    \\[1mm]
			\raisebox{0.15ex}{1}/\raisebox{-0.35ex}{32}                 & 4.6832e-03      & 0.86     & 1.6631e-02      & 0.67     & 2.7754e-02     & 0.56     \\
			\hline
	\end{tabular}}
\end{table*}

We also plot the trajectories of the reference solutions when $t=1$ with different Hurst parameter $H$. It can be seen from Fig. \ref{figure1} that when $H$ tends to small, the trajectories become rough and the fluctuations increase. This observation meets the theoretical expectation since the small $H$ corresponds to the driven noise with low regularity. 
\begin{figure}[!htp]
	\begin{center}
		\includegraphics[scale = 0.33]{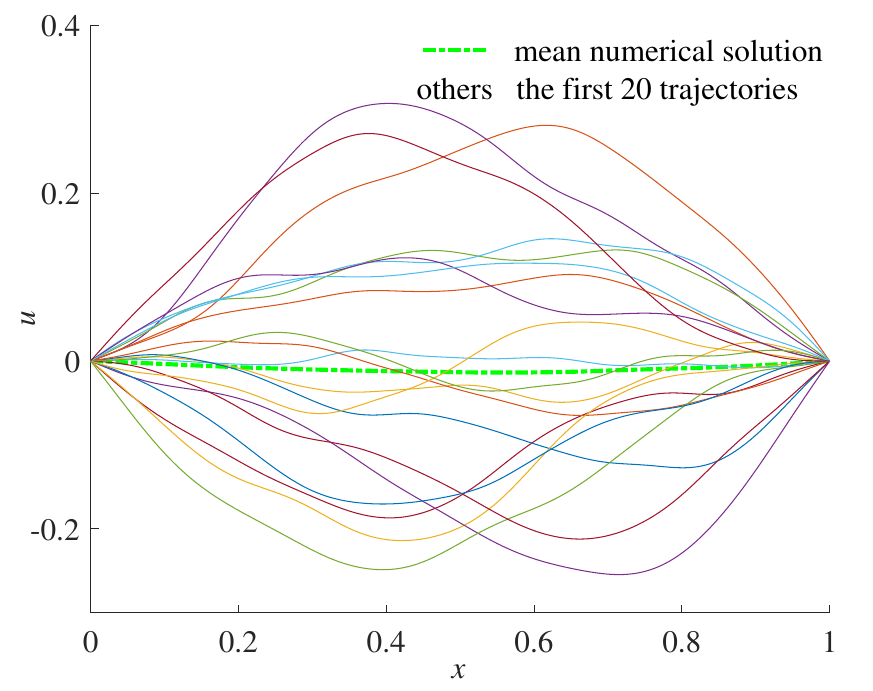} 
		\includegraphics[scale = 0.33]{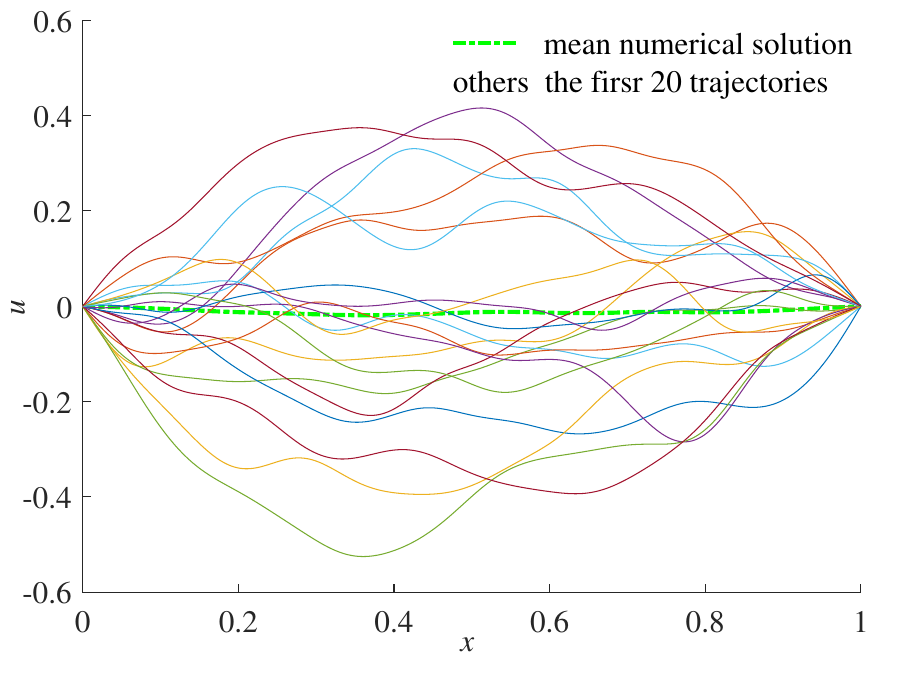}
		\includegraphics[scale = 0.33]{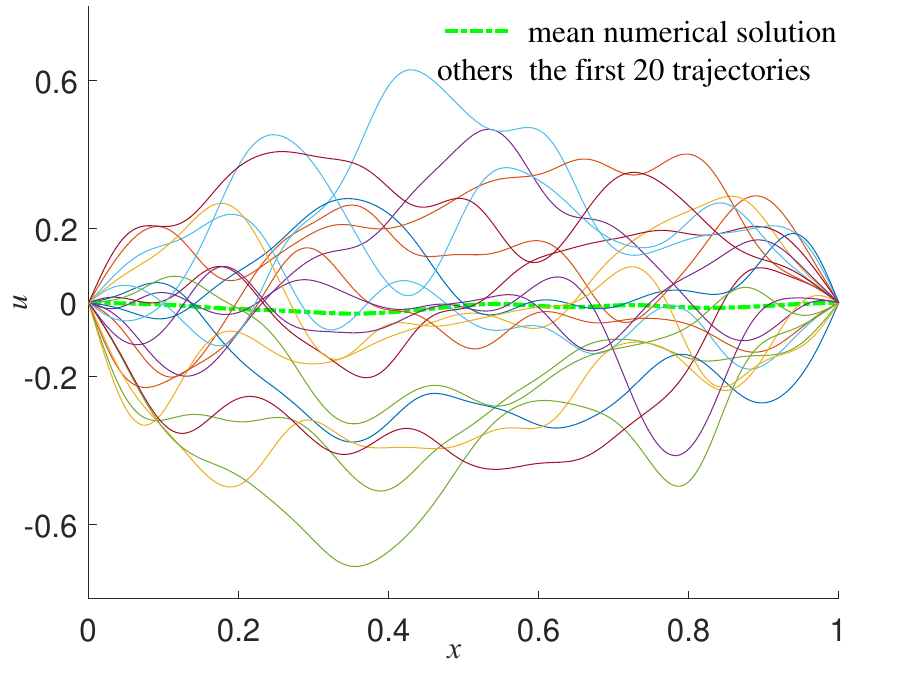}
	\end{center}
	\caption{Left: $H=0.95$. Middle: $H=0.7$. Right: $H=0.55$.  }\label{figure1}
\end{figure}

\section{Summary and discussion}\label{Sec:conclusion}
In this paper, we propose an explicit fully discrete numerical scheme to solve the SBE with cylindrical fBm. Our scheme uses the spectral method for spatial discretization and the tamed exponential integrator method for temporal discretization. To analyze the strong convergence of the scheme, we first prove its convergence in probability by using the stopping time technique. We then show that the full-discrete solution is strongly convergent to the mild solution of the original problem, thanks to the moment boundedness of the full-discrete solution guaranteed by the exponential integrability of the stochastic convolution.

Although the convergence rate is not determined theoretically, it can be observed in numerical tests. The lack of global monotonicity conditions and the low regularity of the cylindrical fBm type noise make it a challenge to determine the strong convergence rate of numerical methods for the SBE. However, for practical applications, numerical solutions obtained by convergent schemes are sufficient to understand the physical meaning behind mathematical models intuitively. Rather than providing the strong convergence rate theoretically, it is more important to demonstrate whether the numerical schemes can preserve the long-time behaviors of the SBE, such as the invariant measure and the ergodicity, which will be the focus of our next effort.

\bibliographystyle{plain}

\end{document}